\documentclass[11pt]{amsart}
\usepackage{amsmath, amsfonts, amssymb, amsthm, amscd}

\usepackage{color}
\usepackage{hyperref}
\usepackage{mathtools}
\usepackage[capitalize,nameinlink,noabbrev,nosort]{cleveref}
\usepackage{graphicx}
\usepackage{tikz}

\usepackage[leqno]{amsmath}
\makeatletter
\newcommand{\leqnomode}{\tagsleft@true}
\newcommand{\reqnomode}{\tagsleft@false}
\makeatother

\usepackage{cases}
\usepackage{subfigure}
\usepackage{color}
\usepackage{latexsym, longtable}
\usepackage{enumerate}
\usepackage{comment}

\usepackage[T1]{fontenc}
\usepackage{tikz-cd}

\newtheorem{theorem}{Theorem}[section]
\newtheorem{proposition}[theorem]{Proposition}

\newtheorem{lemma}[theorem]{Lemma}

\newtheorem{corollary}[theorem]{Corollary}

\newtheorem{definition}[theorem]{Definition}
\newtheorem{defn}[theorem]{Definition}

\newtheorem{example}[theorem]{Example}
\newtheorem{algorithm}[theorem]{Algorithm}

\newtheorem{remark}[theorem]{Remark}
\numberwithin{equation}{section}

\DeclareMathOperator{\PDS}{PDS}
\DeclareMathOperator{\Sym}{Sym}

\DeclareMathOperator{\perm}{perm}

\DeclareMathOperator{\ST}{ST}

\newcommand{\T}{\mathcal{T}}
\newcommand{\Z}{\mathbb{Z}}

\DeclareMathOperator{\arm}{arm}
\DeclareMathOperator{\leg}{leg}
\DeclareMathOperator{\Des}{Des}

\DeclareMathOperator{\maj}{maj}

\DeclareMathOperator{\inv}{inv}
\DeclareMathOperator{\inc}{inc}
\DeclareMathOperator{\coinv}{coinv}

\DeclareMathOperator{\wt}{wt}
\DeclareMathOperator{\dec}{dec}
\DeclareMathOperator{\dg}{dg}

\DeclareMathOperator{\South}{South}

\newlength\cellsize 
\savebox2{%
\begin{picture}(12,12)
\put(0,0){\line(1,0){12}}
\put(0,0){\line(0,1){12}}
\put(12,0){\line(0,1){12}}
\put(0,12){\line(1,0){12}}
\end{picture}}
\newcommand\cellify[1]{\def\thearg{#1}\def\nothing{}%
\ifx\thearg\nothing
\vrule width0pt height\cellsize depth0pt\else
\hbox to 0pt{\usebox2\hss}\fi%
\vbox to 12\unitlength{
\vss
\hbox to 12\unitlength{\hss$#1$\hss}
\vss}}
\newcommand\tableau[1]{\vtop{\let\\=\cr
\setlength\baselineskip{-16000pt}
\setlength\lineskiplimit{16000pt}
\setlength\lineskip{0pt}
\halign{&\cellify{##}\cr#1\crcr}}}
\savebox3{%
\begin{picture}(12,12)
\put(0,0){\line(1,0){12}}
\put(0,0){\line(0,1){12}}
\put(12,0){\line(0,1){12}}
\put(0,12){\line(1,0){12}}
\end{picture}}
\newcommand\expath[1]{%
\hbox to 0pt{\usebox3\hss}%
\vbox to 12\unitlength{
\vss
\hbox to 12\unitlength{\hss$#1$\hss}
\vss}}

\newcommand\cell[3]{
\def\i{#1} \def\j{#2} \def\entry{#3}
\draw (\j-1,-\i)--(\j,-\i)--(\j,-\i+1)--(\j-1,-\i+1)--(\j-1,-\i);
\node at (\j-.5,-\i+.5) {\entry};
}

\newcommand\graycell[3]{
\def\i{#1} \def\j{#2} \def\entry{#3}
\draw[fill=gray!50](\j-1,-\i)--(\j,-\i)--(\j,-\i+1)--(\j-1,-\i+1)--(\j-1,-\i);
\node at (\j-.5,-\i+.5) {\entry};
}

\newcommand\circleT[5]{
\def \n {5}
\def \radius {.5cm}
\def \margin {20}
\edef\s{0}
 \pgfmathparse{\s+1};
\foreach \en in {#1,#2,#3,#4,#5} {%
  \node at ({360/\n * (\s - 1)}:\radius) {$\en$};
   \draw[>=latex] ({360/\n * (\s - 1)+\margin}:\radius) 
     arc ({360/\n * (\s - 1)+\margin}:{360/\n * (\s)-\margin}:\radius);
  \pgfmathparse{\s+1}
      \xdef\s{\pgfmathresult}
}
}

\newcommand\circleS[4]{
\def \n {4}
\def \radius {.5cm}
\def \margin {20}
\edef\s{0}
 \pgfmathparse{\s+1};
\foreach \en in {#1,#2,#3,#4} {%
  \node at ({360/\n * (\s - 1)}:\radius) {$\en$};
   \draw[>=latex] ({360/\n * (\s - 1)+\margin}:\radius) 
     arc ({360/\n * (\s - 1)+\margin}:{360/\n * (\s)-\margin}:\radius);
  \pgfmathparse{\s+1}
      \xdef\s{\pgfmathresult}
}
}

\newcommand{\I}{\mathfrak{S}}

\renewcommand{\arraystretch}{2}

\usepackage[backend=bibtex, giveninits=true, maxcitenames=5]{biblatex}
\title[Compact formulas for  Macdonald polynomials]{Compact formulas for Macdonald polynomials \\ and quasisymmetric Macdonald polynomials}

\date{\today}
\author[Corteel]{Sylvie Corteel}
\address{Department of Mathematics, UC Berkeley, USA}
\email{corteel@berkeley.edu}
\author[Haglund]{Jim Haglund}
\address{Department of Mathematics, University of Pennsylvania, USA}
\email{haglund@math.upenn.edu}
\thanks{Jim Haglund was partially supported by NSF grant DMS-1600670.}
\author[Mandelshtam]{Olya Mandelshtam}
\address{Department of Mathematics, Brown University USA}
\email{olya@math.brown.edu}
\thanks{Olya Mandelshtam was partially supported by NSF grant DMS-1704874.}
\author[Mason]{Sarah Mason}
\address{Department of Mathematics, Wake Forest University, USA}
\email{masonsk@wfu.edu}
\author[Williams]{Lauren Williams}
\address{Department of Mathematics, Harvard University, USA}
\email{williams@math.harvard.edu}
\thanks{Lauren Williams was partially supported by NSF grant DMS-1854512.}

\begin{document}
\begin{abstract}
We present several new and compact formulas for the modified and integral form of the Macdonald polynomials, building on the compact ``multiline queue'' formula for Macdonald polynomials due to Corteel, Mandelshtam and Williams. We also introduce a new quasisymmetric analogue of Macdonald polynomials.  These ``quasisymmetric Macdonald polynomials" refine the (symmetric) Macdonald polynomials and specialize to the quasisymmetric Schur polynomials defined by Haglund, Luoto, Mason, and van Willigenburg. 
\end{abstract}

\maketitle

\section{Introduction}

The symmetric 
\emph{Macdonald polynomials} $P_{\lambda}(X; q, t)$ \cite{Macdonald}
are a family of polynomials in $X = \{x_1, x_2,\dots \}$ indexed by partitions, whose coefficients depend on two parameters $q$ and $t$.  Macdonald polynomials generalize 
multiple important families of polynomials, including Schur polynomials and Hall-Littlewood polynomials. They can be defined as the unique monic basis for the ring of symmetric functions that satisfies certain triangularity and orthogonality conditions. The related \emph{nonsymmetric Macdonald polynomials} $E_{\mu}(X;q,t)$ \cite{Mac88, Macdonald, Cher1} were developed shortly after the introduction of Macdonald polynomials as a tool to study Macdonald polynomials. The $E_{\mu}(X;q, t)$ are indexed by weak compositions and form a basis for the full polynomial ring $\mathbb{Q}[X](q,t)$. 

There has been a great deal of work devoted to understanding Macdonald polynomials from a combinatorial point of view. Haglund-Haiman-Loehr \cite{HHL04} gave a combinatorial formula for the \emph{integral forms} $J_{\lambda}(X; q,t)$, which are scalar multiples of the classical monic forms $P_{\lambda}(X;q, t)$.  They also gave a formula for the nonsymmetric Macdonald polynomials $E_{\mu}(X;q,t)$ \cite{HHL08}, and for the \emph{transformed} or \emph{modified} Macdonald polynomials $\widetilde{H}_{\lambda}(X;q,t)$, which are obtained from $J_{\lambda}(X;q, t)$ via \emph{plethysm}. Macdonald conjectured and Haiman proved \cite{Hai99}, using the geometry of the Hilbert scheme, that the modified Macdonald polynomials $\widetilde{H}_{\lambda}(X;q,t)$ have a positive Schur expansion whose coefficients are $qt$-Kostka polynomials.  However, it is still an open problem to give a combinatorial proof of the Schur positivity or a manifestly positive formula for the $qt$-Kostka polynomials.

Recently, a beautiful connection has been found between Macdonald polynomials and a statistical mechanics model called the multispecies \emph{asymmetric simple exclusion process} (ASEP) on a circle. The ASEP is a one-dimensional exactly solvable particle model; Cantini-deGier-Wheeler \cite{CGW} showed that the partition function of the multispecies ASEP on a circle
is equal to a Macdonald polynomial $P_{\lambda}(x_1,\dots,x_n; q, t)$ evaluated at $q=1$ and $x_i=1$ for all $i$. Building on this result as well as work of Martin \cite{Martin}, the first, third, and fifth authors recently used \emph{multiline queues} to simultaneously compute the stationary probabilities of the multispecies exclusion process, and give compact formulas for the symmetric Macdonald polynomials $P_{\lambda}$ and the nonsymmetric Macdonald polynomials  $E_{\lambda}$ \cite{CMW18}, for any partition $\lambda$. These formulas are ``compact'' in that they have fewer terms than the formulas of Haglund-Haiman-Loehr.

In this paper, we use the above ideas to continue the search for compact formulas for Macdonald polynomials. Our first two main results are compact formulas for the modified Macdonald polynomials $\widetilde{H}_{\lambda}(X;q,t)$ and the integral forms $J_{\lambda}(X;q,t)$; these new formulas have far fewer terms than other known combinatorial formulas. Our third main result uses the connection with the ASEP on a ring towards a different application: the introduction of a new family of quasisymmetric functions we call \emph{quasisymmetric Macdonald polynomials}  $G_{\gamma}(X; q, t)$. We show that $G_{\gamma}(X; q, t)$ is indeed a quasisymmetric function, and give a combinatorial formula for the $G_{\gamma}(X; q, t)$ corresponding to ``pieces'' of the compact formula for the $P_{\lambda}(X; q, t)$ from \cite{CMW18}. Moreover, the quasisymmetric functions $G_{\gamma}(X; q, t)$ at $q=t=0$ specializes to the \emph{quasisymmetric Schur functions} $\text{QS}_{\gamma}(X)$ introduced by the second and fourth authors, together with Luoto and van Willigenburg \cite{HLMV09}.  The quasisymmetric Schur functions form a basis for the ring of quasisymmetric functions, and until now it has been an open question to find a refinement of the Macdonald polynomials $P_{\lambda}$ into quasisymmetric pieces which generalize the quasisymmetric Schur functions.

We note that Garbali and Wheeler have recently used
integrable lattice models to give a new formula for the modified
Macdonald polynomials \cite{GarbaliWheeler}.   Their formula can be viewed as an extension of the ``fermionic formula" of Kerov, Kirillov and Reshetikhin \cite{KKR86}, \cite{KR88} for modified Hall-Littlewood polynomials, and the combinatorics involved is quite a bit different than that of the HHL formula or of our compact version.

This paper is organized as follows. In \cref{sec:def}, we provide the relevant background. \cref{sec:Hcompact} and \cref{sec:Jcompact} describe our two compact formulas.  \cref{sec:qsym} defines our new quasisymmetric Macdonald polynomials as well as several open problems which naturally arise from this work.

\textsc{Acknowledgements:}
This material is based upon work supported by the National Science Foundation
under agreement No.\ DMS-1600670, No.\ DMS-1704874,  and No.\ DMS-1854512.  Any opinions,
findings and conclusions or recommendations expressed in this material
are those of the authors and do not necessarily reflect  the
views of the National Science Foundation. This material is based upon work supported by the Swedish Research Council under grant no.
2016-06596 while the first and third authors were in residence at Institut Mittag-Leffler in Djursholm, Sweden
during Spring 2020.

\section{Definitions}\label{sec:def}

We begin by introducing the 
notation and definitions we will need in Sections \ref{sec:Hcompact} and \ref{sec:Jcompact}.  In our partition and composition diagrams, the columns are labeled from left to right, and the rows are labeled from bottom to top, so that the notation $(i,r)$ refers to the box (or \emph{cell}) in the $i^{th}$ column from the left and the $r^{th}$ row from the bottom. 

Given a partition or composition $\alpha=(\alpha_1,\dots, \alpha_n)$, 
its \emph{diagram} $\dg(\alpha)$ is a sequence of columns (bottom justified), where the $i^{th}$ column has $\alpha_i$ cells.  See 
\cref{fig:leg} for an example.
  Note that the Ferrers graph of a partition 
  $\lambda = (\lambda_1,\dots,\lambda_n)$ 
  is typically defined to contain $\lambda _i$ squares in the $i$th row (from the bottom).  This means that when $\alpha$ is a partition, the diagram $\dg(\alpha)$ is the Ferrers graph of the conjugate partition $\alpha^{\prime}$.

For a cell $s\in \dg(\lambda)$, let $\sigma(s)$ denote the integer assigned to $s$, i.e., the integer occupying cell $s$.  The numbers appearing in such a filling are called the \emph{entries}. 

The \emph{leg} of a cell $(i,r)$ 
is the set of cells in column $i$ above the cell $(i,r)$. 
The \emph{arm} of a cell $(i,r)$ 
is the set of cells in row $r$ to the right of the cell $(i,r)$ contained in a column whose height doesn't exceed $\alpha_i$, together with the cells in row $r-1$ to the left of the cell $(i,r-1)$, contained in a column whose height is smaller than $\alpha_i$. 
We let $\leg((i,r))$
and $\arm((i,r))$ denote the number of cells in the leg and arm of the cell $(i,r)$,
respectively.  
See \cref{fig:leg}.

\begin{figure}
  \centering
\begin{tikzpicture}[scale=.5]
\cell31{$$} \cell32{} \cell33{$$} \cell34{$u$}  \cell37{$a$}\cell39{$a$}
\cell41{} \cell42{} \cell43{$a$}\cell44{}  \cell46{$$} \cell47{} \cell49{$$} 
\cell12{} 
\cell21{} \cell22{} \cell24{$\ell$} \cell27{}
\cell51{$$} \cell52{$$}  \cell53{$$} \cell54{$$}  \cell55{$$}  \cell56{$$}  \cell57{$$}  \cell58{$$}  \cell59{$$} 
\end{tikzpicture} 
\caption{The diagram of the composition $(4,5,3,4,1,2,4,1,3)$  and the cells in the leg and the arm of the cell $u=(4,3)$. Here $\leg(u)=1$ and $\arm(u)=3$.}
\label{fig:leg}
\end{figure}
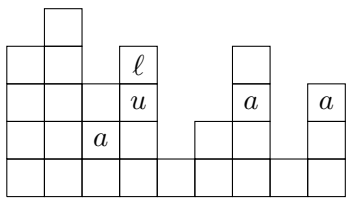

Given a partition $\lambda$, a \emph{filling} $\sigma:\dg(\lambda) \to \Z^+$ is an assignment of positive integers to the cells of the diagram of $\lambda$ and is denoted by $\sigma$. The \emph{$x$-weight} of a filling $\sigma$ of $\dg(\lambda)$
is defined in a similar fashion to the $x$-weight of a semistandard Young tableau, namely \[x^{\sigma} = \prod_{s\in \dg(\lambda)} x_{\sigma(s)}.\]

We recall several definitions from~\cite{HHL04}. We augment the diagram of a 
composition or partition $\alpha$ by adding a \emph{basement}, i.e., a zero'th row of size equal to the number of parts of $\alpha$.
In this section we will fill all cells of the basement
 with the entry $\infty$; later in 
\cref{sec:Jcompact}
  we will use a more 
 general definition of basement.

\begin{definition}
	Given a  partition $\lambda$ and a filling $\sigma$ of $\dg(\lambda)$, a \emph{triple} consists of three cells $(v,r)$, $(u,r)$, and $(u,r-1)$ in the diagram, where $u,v$, and $r$ are positive integers with $u<v$.  Let $a = \sigma((v,r))$, $b = \sigma((u,r))$, and $c=\sigma((u,r-1))$.  These entries are therefore arranged in the configuration $\tableau{b\\  c}\quad \tableau{a}$\,. We say that the triple  is a \emph{counterclockwise triple}, or an \emph{inversion triple}, if the entries are increasing in counterclockwise order.  If two entries are equal, we consider the one appearing first when the entries are read top-to-bottom, left-to-right to be the smaller entry. In other words, it is an inversion triple if any of the following holds:
\begin{align*}
a < b \leq c\ \ {\rm or}\ \ 
c < a < b \ \ {\rm or}\ \ 
b \leq c < a.
\end{align*}
In all other cases we say the triple is a \emph{clockwise triple}, or a \emph{coinversion triple}. 
\end{definition}

For example, in \cref{fig:perm}, the cells $(3,3)$, $(1,3)$ and $(1,2)$  form a (counterclockwise) inversion triple.

Note that since $\sigma((j,0))=\infty$, for $u<v$, we have that $(v,1)$ and  $(u,1)$ form a (counterclockwise) inversion triple if and only if $\sigma((u,1))>\sigma((v,1))$. In this case we say that the triple is a {\em degenerate inversion triple}. For example, in \cref{fig:perm} the cells $(3,1)$ and $(6,1)$  form a degenerate inversion triple. Given a filling $\sigma$ of  $\dg(\lambda)$, let $\inv(\sigma)$ be the total number of  (counterclockwise) inversion triples, including degenerate inversion triples.

Define the \emph{set of descents} of a filling $\sigma$ of $\dg(\lambda)$ to be 
\[
	\Des(\sigma) = \{(u,r)\in \dg(\lambda)\ :\ \sigma((u,r))>\sigma((u,r-1))\}.
\]   
The \emph{major index} $\maj(\sigma)$ is then defined to be the sum of the legs of the descents of 
	$\sigma$ plus the number of descents:
\[
\maj(\sigma) = \sum_{s \in \Des(\sigma)} (\leg(s)+1).
\]
The filling $\sigma$ of $\dg(\lambda)$ in \cref{fig:perm} has $\inv(\sigma)=22$ and $\maj(\sigma)=5$. 

\begin{definition}\label{def:weak}
Given a \emph{weak composition}, i.e.\ a vector $\alpha = (\alpha _1, \alpha _2, \ldots , \alpha _n)$ of nonnegative integers, we let $\inc(\alpha)$ and $\dec(\alpha)$ be the vectors obtained from $\alpha$ by sorting the parts into weakly increasing order, and weakly decreasing order, respectively. Let $\beta (\alpha)$ be the permutation in $S_n$ of maximal length with the property that $\beta$ applied to the entries of the vector $\alpha$ yields $\inc(\alpha)$.  Let $\alpha ^{+}$ be the \emph{strong composition} obtained from $\alpha$ by removing  the zeros and let $\ell (\alpha)$ be the number of parts of $\alpha ^{+}$.  For example, if $\alpha = (0,2,0,2,1, 3)$ then  $\inc(\alpha) = (0,0,1,2,2,3)$, $\dec(\alpha) = (3,2,2,1,0,0)$, $\beta (\alpha) = (3,1,5,4,2,6)$, $\alpha ^{+} = (2,2,1,3)$, and $\ell (\alpha) = 4$.     
\end{definition}

\section{Compact formula for modified Macdonald polynomials}\label{sec:Hcompact}

Our first main result is a ``compact'' formula for the modified Macdonald
polynomials  $\widetilde{H}_{\lambda}(X; q, t)$, given in 
\cref{thm:MacPerm2}. 
  Before explaining
our result, we first
recall the combinatorial formula 
 of Haiman, Haglund and Loehr \cite{HHL04}.
\begin{theorem}[{\cite[Theorem 2.2]{HHL04}}] \label{thm:HHL} 
Let $\lambda$ be a partition.
	The modified Macdonald polynomial $\widetilde{H}_{\lambda}(X;q,t)$ is given by
	\begin{equation}\label{eq:HHL}
        \widetilde{H}_{\lambda}(X;q,t) = \sum_{\sigma: \text{dg}(\lambda) \to \Z^{+}} x^\sigma q^{\maj(\sigma)}t^{\inv(\sigma)},
	\end{equation}
where the sum is over all fillings of $\dg(\lambda)$.
\end{theorem}
\begin{remark}
The careful reader might notice that 
\cite[Theorem 2.2]{HHL04} looks different from 
	\eqref{eq:HHL}, because the roles of $q$ and $t$
are swapped, and the diagram of $\lambda$ as defined in 
	\cite{HHL04}
is the conjugate of the diagram of $\lambda$ we defined here.
  However, because of the 
$qt$-symmetry of Macdonald polynomials -- 
	that is, $\widetilde{H}_{\lambda}(X;q,t) = \widetilde{H}_{\lambda'}(X;t,q)$\footnote{This follows from the triangularity conditions defining 
	$\widetilde{H}_{\lambda}$ and also follows from 
	Haiman's geometric interpretation \cite{HaimanHilbert}.} --
	\eqref{eq:HHL} is equivalent to 
	\cite[Theorem 2.2]{HHL04}.
\end{remark}

While \cref{thm:HHL} is simple and elegant, it has the disadvantage of containing many terms, since it is a sum over \emph{all} fillings of  $\dg(\lambda)$ by positive integers. By contrast, our compact formula (\cref{thm:MacPerm2}) is a sum over far fewer terms---it is a sum over \emph{sorted tableaux}.  To define these sorted tableaux, we first define an order on the columns of the fillings.

\begin{figure}
  \centering
\begin{tikzpicture}[scale=.5]

\node at (-1.5,.5) {\rm row 5};
\node at (-1.5,-.5) {\rm row 4};
\node at (-1.5,-1.5) {\rm row 3};
\node at (-1.5,-2.5) {\rm row 2};
\node at (-1.5,-3.5) {\rm row 1};
\node at (-1.5,-4.5) {\rm row 0};

\cell315 \cell325 \cell336 \cell346 \cell356
\cell215 \cell225 \cell232  
\cell419 \cell429 \cell439\cell441 \cell453 \cell462 \cell472 \cell483 \cell493
\cell111 \cell121\cell131
\cell016\cell026\cell036
\cell51{$\infty$} \cell52{$\infty$}  \cell53{$\infty$} \cell54{$\infty$}  \cell55{$\infty$}  \cell56{$\infty$}  \cell57{$\infty$}  \cell58{$\infty$}  \cell59{$\infty$} 
\draw (0,-4)--(0,1)--(3,1)--(3,-2)--(5,-2)--(5,-3)--(9,-3);

\end{tikzpicture} 
	\caption{A filling in which the cells $(3,3), (1,3)$ and $(1,2)$
	form an inversion triple, and the cells $(3,1)$ and $(6,1)$ form a degenerate inversion triple.  This filling is a sorted tableau, with 
	$\perm_t(\sigma) = {3 \choose 2,1}_t {2 \choose 1,1}_t {4 \choose 2,2}_t$.}
\label{fig:perm}
\end{figure}

\begin{definition}\label{def:orderhhl}
Fix a filling $\sigma$ of a diagram,
	and consider two columns $A$ and $B$ of height $j$, with $A$ to the 
	left of $B$. Let $a_1,\dots, a_j$ and $b_1,\dots, b_j$ be the entries of columns $A$ and $B$, respectively, read from bottom to top. We say that $A \lhd B$, if either $a_1 < b_1$, or $a_i = b_i$ for $i=1, 2,  \dots, r-1$ (for some positive $r$), and the cells containing $b_r$, $a_r$ and $a_{r-1}$ do not form a counterclockwise (inversion) triple.
\begin{center}
\begin{tikzpicture}[scale=.7]
\node at (.55,1.5) {$A\quad \lhd \quad B$};
\node at (.55,-1.5) {$\neq$};
\node at (.55,-2.5) {$=$};
\node at (.55,-4.5) {$=$};
\cell50{$a_1$}
\cell5{2}{$b_1$};
\cell40{$\vdots$}; 
\cell4{2}{$\vdots$};
\graycell30{$a_{r-1}$};
\cell3{2}{$b_{r-1}$};
\graycell20{$a_{r}$}; 
\graycell2{2}{$b_{r}$}; 
\cell10{$\vdots$}; 
\cell1{2}{$\vdots$};
\cell00{$a_j$};
\cell0{2}{$b_j$};

\end{tikzpicture}
\end{center}
\end{definition}

\begin{definition}\label{def:ptableau2}
	Given a filling $\sigma$ of the diagram of a partition $\lambda$, we say that $\sigma$ is a \emph{sorted tableau of $\dg(\lambda)$} if, for all positive integers $h$, when we read all columns of height $h$ from left to right, the columns appear in weakly increasing order with respect to $\lhd$.
See \cref{fig:perm} for an example.
We write $\ST(\lambda)$ for the set of all sorted tableaux of $\dg(\lambda)$.
\end{definition}

Let $\sigma$ be a sorted tableau. First suppose that the diagram of $\lambda$ is an $m \times n$ rectangle. The $n$ columns may not all have distinct fillings: suppose that among those $n$ columns, there are $j$ distinct column fillings, with $u_1$ identical columns of the first filling, $u_2$ identical columns of the second filling, \dots, $u_j$ identical columns all containing the $j$th filling. Define 
\begin{equation}\label{eq:perm}
\perm_t(\sigma) = {n \choose{u_1,\ldots,u_j}}_t,
\end{equation}
the $t$-analogue of the multinomial coefficient ${n \choose {u_1,\ldots,u_j}}$.
Now suppose $\sigma$ is a sorted tableau which is a concatenation of rectangular sorted tableaux $\sigma_1,\dots, \sigma_{\ell}$, all of different heights.  Define $\perm_t(\sigma) = \prod_{i=1}^{\ell} \perm_t(\sigma_i).$ 

Our main result in this section is a  compact formula for $\widetilde{H}_{\lambda}$.
\begin{theorem}{\label{thm:MacPerm2}}
Let $\lambda$ be a partition.  The modified Macdonald polynomial
	$\widetilde{H}_{\lambda}(X;q,t)$ 
	equals 
	\[
	\widetilde{H}_{\lambda}(X;q,t) = \sum_{\sigma\in \ST(\lambda)} x^\sigma t^{\inv(\sigma)} q^{\maj(\sigma)} 
	\perm_t(\sigma),
\]
	where the sum is over all sorted tableaux of $\dg(\lambda)$.
\end{theorem}

\begin{example} We use \cref{thm:MacPerm2} to compute
	$\widetilde{H}_{(2,1,1)}(x_1, x_2, x_3;q,t)$
 in \cref{fig:alternativeDef}.
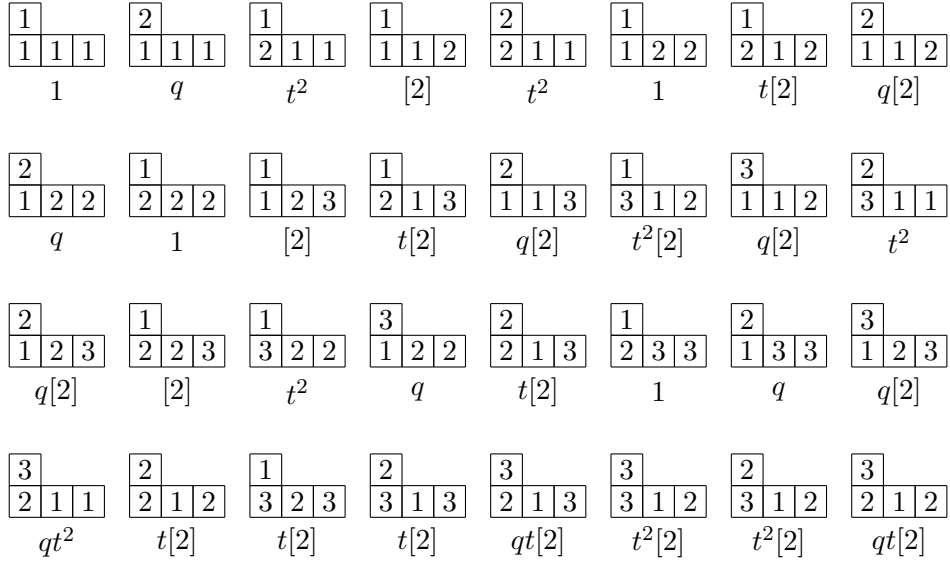
\begin{figure}
\begin{center}
\begin{tikzpicture}[node distance=2cm]
\def \sh {1.6};
\node at (0,0) {$\tableau{ 1\\1&1&1 }$};
\node at (0,-.75) {$1$};
\node at (\sh,0) {$\tableau{ 2\\1&1&1 }$};
\node at (\sh,-.75) {$q$};
\node at (2*\sh,0) {$\tableau{ 1\\2&1&1 }$};
\node at (2*\sh,-.75) {$t^2$};
\node at (3*\sh,0) {$\tableau{ 1\\1&1&2 }$};
\node at (3*\sh,-.75) {$[2]$};
\node at (4*\sh,0) {$\tableau{ 2\\2&1&1 }$};
\node at (4*\sh,-.75) {$t^2$};
\node at (5*\sh,0) {$\tableau{ 1\\1&2&2 }$};
\node at (5*\sh,-.75) {$1$};
\node at (6*\sh,0) {$\tableau{ 1\\2&1&2 }$};
\node at (6*\sh,-.75) {$t[2]$};
\node at (7*\sh,0) {$\tableau{ 2\\1&1&2 }$};
\node at (7*\sh,-.75) {$q[2]$};
\node at (0,-2) {$\tableau{ 2\\1&2&2 }$};
\node at (0,-2.75) {$q$};
\node at (\sh,-2) {$\tableau{ 1\\2&2&2 }$};
\node at (\sh,-2.75) {$1$};
\node at (2*\sh,-2) {$\tableau{ 1\\1&2&3 }$};
\node at (2*\sh,-2.75) {$[2]$};
\node at (3*\sh,-2) {$\tableau{ 1\\2&1&3 }$};
\node at (3*\sh,-2.75) {$t[2]$};
\node at (4*\sh,-2) {$\tableau{ 2\\1&1&3 }$};
\node at (4*\sh,-2.75) {$q[2]$};
\node at (5*\sh,-2) {$\tableau{ 1\\3&1&2 }$};
\node at (5*\sh,-2.75) {$t^2[2]$};
\node at (6*\sh,-2) {$\tableau{ 3\\1&1&2 }$};
\node at (6*\sh,-2.75) {$q[2]$};
\node at (7*\sh,-2) {$\tableau{ 2\\3&1&1 }$};
\node at (7*\sh,-2.75) {$t^2$};

\node at (0,-4) {$\tableau{ 2\\1&2&3 }$};
\node at (0,-4.75) {$q[2]$};
\node at (\sh,-4) {$\tableau{ 1\\2&2&3 }$};
\node at (\sh,-4.75) {$[2]$};
\node at (2*\sh,-4) {$\tableau{ 1\\3&2&2 }$};
\node at (2*\sh,-4.75) {$t^2$};
\node at (3*\sh,-4) {$\tableau{ 3\\1&2&2 }$};
\node at (3*\sh,-4.75) {$q$};
\node at (4*\sh,-4) {$\tableau{ 2\\2&1&3 }$};
\node at (4*\sh,-4.75) {$t[2]$};
\node at (5*\sh,-4) {$\tableau{ 1\\2&3&3 }$};
\node at (5*\sh,-4.75) {$1$};
\node at (6*\sh,-4) {$\tableau{ 2\\1&3&3 }$};
\node at (6*\sh,-4.75) {$q$};
\node at (7*\sh,-4) {$\tableau{ 3\\1&2&3 }$};
\node at (7*\sh,-4.75) {$q[2]$};

\node at (0*\sh,-6) {$\tableau{ 3\\2&1&1 }$};
\node at (0*\sh,-6.75) {$qt^2$};
\node at (1*\sh,-6) {$\tableau{ 2\\2&1&2 }$};
\node at (1*\sh,-6.75) {$t[2]$};
\node at (2*\sh,-6) {$\tableau{ 1\\3&2&3 }$};
\node at (2*\sh,-6.75) {$t[2]$};
\node at (3*\sh,-6) {$\tableau{ 2\\3&1&3 }$};
\node at (3*\sh,-6.75) {$t[2]$};
\node at (4*\sh,-6) {$\tableau{ 3\\2&1&3 }$};
\node at (4*\sh,-6.75) {$qt[2]$};
\node at (5*\sh,-6) {$\tableau{ 3\\3&1&2 }$};
\node at (5*\sh,-6.75) {$t^2[2]$};
\node at (6*\sh,-6) {$\tableau{ 2\\3&1&2 }$};
\node at (6*\sh,-6.75) {$t^2[2]$};
\node at (7*\sh,-6) {$\tableau{ 3\\2&1&2 }$};
\node at (7*\sh,-6.75) {$qt[2]$};

\end{tikzpicture}
\end{center}
        \caption{We compute
		$\widetilde{H}_{(2,1,1)}(x_1,x_2,x_3;q,t)$ by 
		adding the weights of the sorted tableaux of 
		$\dg((2,1,1))$. In the figure above, we've listed $t^{\inv(\sigma)} q^{\maj(\sigma)} \perm_t(\sigma)$ below each sorted tableau $\sigma$, but have omitted $x^{\sigma}$ to save space.  Here $[i]$ denotes $[i]_t$. }\label{fig:alternativeDef}
\end{figure}
\end{example}

To prove \cref{thm:MacPerm2}, we use the \emph{inversion flip operators}
of Loehr and Niese~\cite{LoeNie12}, 
which act on fillings of a given diagram. These operators fix the $\maj$ statistic and change the $\inv$ statistic by one; in other words, they change the number of counterclockwise inversion triples  by one.

\begin{definition}
	Let $\sigma$ be a filling with columns $i,i+1$ having different entries at row $r$, and identical entries $\sigma(i,s)=\sigma(i+1,s)$
	in all rows $s$ below $r$ (i.e., $s<r$). The operator $\mathcal{T}_i^{(r)}$ consists of two steps.  The first step swaps the entries in the cells $(i,r)$ and $(i,r+1)$.
The second step adjusts for any changes to triples that might have happened in higher rows due to this swapping.  The precise statement of these steps is as follows.
\begin{enumerate}
\item Switch the entries $\sigma(i,r)$ and $\sigma(i+1,r)$.
\item Consider the triple formed by the cells $(i+1,r+1)$, $(i,r+1)$, and $(i,r)$, if it exists.  (If it does not exist, terminate the process.)
\begin{enumerate}
	\item If this triple is a counterclockwise triple before \emph{and} after the swap, terminate the process.
	\item If this triple is a clockwise triple before \emph{and} after the swap, terminate the process.
\item Otherwise, switch the entries $\sigma(i,r+1)$ and $\sigma(i+1,r+1)$ and go to step (2), after replacing $r$ by $r+1$.
\end{enumerate}
\end{enumerate}
\end{definition}

Note that $\mathcal{T}_i^{(r)}$ is well-defined if and only if the entries in columns $i,i+1$ are different at row $r$ and identical at all rows below $r$. Thus for each $i$, there is at most one $r$ such that $\mathcal{T}_i^{(r)}$ is well-defined. However, we keep the $r$ in the superscript to improve the readability of our proofs.

\begin{example}
To apply the operator $\mathcal{T}_1^{(2)}$ to the filling $\sigma$ shown below, first swap the entries in row $2$.  Then notice that the clockwise triple $\tableau{2&5  \\ 1}$  becomes a counterclockwise triple $\tableau{2&5 \\ 4}$.  So we must swap the $2$ and $5$ in row $3$.  At this point the clockwise triple $\tableau{3&4 \\ 2}$ remains a clockwise  triple $\tableau{3&4 \\ 5}$ and so the procedure is complete.

\vspace{0.1in}
\begin{tikzpicture}[scale=2]
\put(130,-25){$\sigma=$};
\put(155,0){\tableau{3 & 3 \\ 2 & 4 \\ 3 & 4 \\ 2 & 5 \\ 1 & 4 \\ 3 & 3 }} ;
\put(185,-25){$\longrightarrow$};
\put(205,0){\tableau{3 & 3 \\ 2 & 4 \\ 3 & 4 \\ 2 & 5 \\ 4 & 1 \\ 3 & 3 }};
\put(235,-25){$\longrightarrow$};
\put(255,0){\tableau{3 & 3 \\ 2 & 4 \\ 3 & 4 \\ 5 & 2 \\ 4 & 1 \\ 3 & 3 }};
\end{tikzpicture}
\end{example}

\vspace{.8in}

\subsection{How the operators affect the $\maj$ and $\inv$ statistics}

\cref{lem:maj} and 
\cref{lem:coinv}
 appear as Theorem 5.3 in 
the work of Loehr and Niese \cite{LoeNie12}.  For completeness, we provide here
a careful proof of each, including details for several cases left to the reader in~\cite{LoeNie12}.

\begin{lemma}\label{lem:involution}\cite[Lemma 5.2]{LoeNie12}.
Each operator $\mathcal{T}_i^{(r)}$ is an involution.
\end{lemma}
\begin{proof}
As in \cite{LoeNie12}, this follows directly from the definition that each operator is an involution.  
\end{proof}

\begin{lemma}\label{lem:maj}
Let $\sigma$ be a filling. Fix row $r$ and column $i$ within $\sigma$.  Assume the entries in columns $i$ and $i+1$ of $\sigma$ are identical in rows $r-1$ and below. Then
 \[\maj(\mathcal{T}_i^{(r)} (\sigma)) = \maj(\sigma).\]
\end{lemma}

\begin{proof}
In what follows, we make the convention that if the cell $(i,r)$ is not part of the diagram of $\sigma$, then $\sigma(i,r)=0$.  
To prove that $\maj$ is fixed under the application of a $\mathcal{T}_i^{(r)}$ operator, note that the contribution to $\maj$ from row $r$ remains the same since $\sigma(i, r-1)=\sigma(i+1, r-1)$.  Consider a higher row $j$ of $\sigma$ in column $i$ ($j\ge r$) and label the entries in rows $j+1$ and $j$ in columns $i$ and $i+1$ as shown below.

\begin{center}
\begin{tikzpicture}[scale=.5]
\cell01{$c$} \cell02{$d$}
\cell11{$a$} \cell12{$b$}

\node at (-2.65,.5) {\rm row $j+1$};
 \node at (-2,-.5) {\rm row $j$};
\node at (.5, 1.5) {$i$};
\node at (1.7, 1.5) {$i+1$};
\end{tikzpicture}
\end{center}

If $a$ and $b$ are not swapped after the application of $\mathcal{T}_i^{(r)}$, then there is nothing to check.  We now consider the possible repercussions of swapping $a$ and $b$.  (Note that we may assume $c \neq d$ since otherwise there would also be nothing to check.  Also, if $a$ and $b$ as well as $c$ and $d$ are swapped, the contributions to $\maj$ from this portion of the filling remain the same.) 

Assume therefore that the operator $\mathcal{T}_i^{(r)}$ sends $\tableau{ c & d \\ a & b}$ to $\tableau{ c & d \\ b & a}$\ ; that is, $c$ and $d$ are not swapped under $\mathcal{T}_i^{(r)}$.  If $c>a$ and $d>b$ but $d \le a$, then $b<d\le a<c$, implying that $d, c, a $ forms a clockwise triple in $\sigma$ while $d, c,b$ forms a counterclockwise triple in $\mathcal{T}_i^{(r)} (\sigma)$, contradicting the conditions of the adjustment process.  Similarly if $c>a$ and $d>b$ but $c \le b$, then $a<c\le b<d$, implying that $a,c,d$ form a clockwise triple in $\sigma$ while $b,c,d$ form a counterclockwise triple in $\mathcal{T}_i^{(r)}(\sigma)$, also a contradiction.

Next assume that $c>a$ but $d \le b$.  If $d > a$ and $c >b$, then $a<d \le b<  c$, implying that $a,d,c$ form a counterclockwise triple in $\sigma$ but $d,b,c$ form a clockwise triple in $\mathcal{T}_i^{(r)}(\sigma)$, a contradiction.  Similarly, if $d \le a$ and $c \le b$, then $d \le a<c \le b$, implying that $d,a,c$ form a clockwise triple in $\sigma$ while $d,c,b$ form a counterclockwise triple in $\mathcal{T}_i^{(r)}(\sigma)$.  An analogous argument shows that when $d>b$ and $c \le a$, the contribution to the $\maj$ statistic remains fixed, as well.

Finally, assume $c \le a$ and $d \le b$.  If $d>a$, then $c \le a < d \le b$, implying that $c,a,d$ form a counterclockwise triple in $\sigma$ while $b,c,d$ form a clockwise triple in $\mathcal{T}_i^{(r)}(\sigma)$ in this collection of cells.  Similarly, if $c > b$, then $d \le b < c \le a$, implying that $d,c,a$ form a counterclockwise triple  in $\sigma$ while $d,b,c$ form a clockwise triple in $\mathcal{T}_i^{(r)}(\sigma)$, letting us conclude that all cases lead to contradictions.  Therefore, although the locations of the descents might change from $\sigma$ to $\mathcal{T}_i^{(r)}(\sigma)$, the total contribution to the major index remains the same.
\end{proof}

\begin{lemma}\label{lem:coinv}
Let $\sigma$ be a filling. Fix row $r$ and column $i$ within $\sigma$.  We assume that the entries in cells $(i,r)$ and $(i+1,r)$ are distinct, and that the entries in columns $i$ and $i+1$ of $\sigma$ are identical in rows $r-1$ and below.  Let $T$ be the (possibly degenerate) triple consisting of the three  cells $(i+1,r), (i,r), (i,r-1)$.  Then 
\[
\inv( \mathcal{T}_i^{(r)} (\sigma)) = \inv(\sigma) + \begin{cases} -1, & \textrm{if } $T$ \textrm{ is a counterclockwise triple}\\
1,& \textrm{otherwise.}
\end{cases}
\]
\end{lemma}

\begin{proof}
	Consider the columns $i$ and ${i+1}$ of $\sigma$ as shown below, where the columns are identical from the bottom to row $r-1$. Suppose also that row $r+j-1$ contains  $c_j \neq d_j$ for $1 \leq j \leq n$ for some $n \geq 1$. 
Suppose $\mathcal{T}_i^{(r)}$ is applied with the first swap occurring at $c_1,d_1$, propagating up to $c_n,d_n$, and finally terminating at row $r+n-1$.

\begin{center}
\begin{tikzpicture}[scale=0.55]
\node at (2,1.7){$\sigma$};
\node at (8,1.7){$\mathcal{T}_i^{(r)}(\sigma)$};
\node at (-3.6,-1.5){row $r+n-1$};
\node at (-2.5,-3.5){row $r$};
\node at (-3.05,-0.5){row $r+n$};
\draw[->](4,-2)--(5.5,-2) node[midway,above] {$\mathcal{T}_i^{(r)}$};
\cell01{$\vdots$}\cell02{$\vdots$}\cell0{3.5}{$\vdots$}
\cell11{$a$}\cell12{$b$}\cell1{3.5}{$f$}
\cell21{$c_n$}\cell22{$d_n$}\cell2{3.5}{$f_n$}
\cell31{$\vdots$}\cell32{$\vdots$}\cell3{3.5}{$\vdots$}
\cell41{$c_1$}\cell42{$d_1$}\cell4{3.5}{$f_1$}
\cell51{$e$}\cell52{$e$}\cell5{3.5}{$\vdots$}

\cell07{$\vdots$}\cell08{$\vdots$}\cell0{9.5}{$\vdots$}
\cell17{$a$}\cell18{$b$}\cell1{9.5}{$f$}
\cell27{$d_n$}\cell28{$c_n$}\cell2{9.5}{$f_n$}
\cell37{$\vdots$}\cell38{$\vdots$}\cell3{9.5}{$\vdots$}
\cell47{$d_1$}\cell48{$c_1$}\cell4{9.5}{$f_1$}
\cell57{$e$}\cell58{$e$}\cell5{9.5}{$\vdots$}

\node at (0.5,-5.5){$i$};
\node at (3,-5.5){$j$};
\node at (6.5,-5.5){$i$};
\node at (9,-5.5){$j$};
\end{tikzpicture}
\end{center}

By construction, the triple $(d_1,c_1,e)$ is counterclockwise in $\sigma$ if and only if $(c_1,d_1,e)$ is clockwise in $\mathcal{T}_i^{(r)}(\sigma)$, where $e=\sigma((i,r-1))$. All other triples in columns $i,i+1$ are counterclockwise in $\sigma$ if and only if the corresponding triples in $\mathcal{T}_i^{(r)}(\sigma)$ are as well.

Swapping entries in columns $i,i+1$ will not affect any triples with entries in columns to the left of $i$. It remains to check that $\mathcal{T}_i^{(r)}$ preserves the total number of counterclockwise triples with an entry in some column $j>i$ and entries in columns $i$ or $i+1$, shown in the figure above. 

Define $f_u=\sigma((j,r+u-1))$ for $u=1,\ldots,n$ and $f=\sigma((j,r+n))$. Both $\sigma$ and $\mathcal{T}_i^{(r)}(\sigma)$ have the same triples $(f_{u},c_{u},c_{u-1})$ and $(f_{u},d_{u},d_{u-1})$ for $u=2,\ldots,n$, along with $(f_1,c_1,e)$ and $(f_1,d_1,e)$. The only triples that change after applying the operator are the triples $(f,a,c_n)$ and $(f,b,d_n)$ in $\sigma$, where $a=\sigma(i,r+n)$ and $b=\sigma(i+1,r+n)$.  These triples become $(f,a,d_n)$ and $(f,b,c_n)$ in $\mathcal{T}_i^{(r)}(\sigma)$. We will now check that the total number of counterclockwise triples is preserved in each of these pairs. To simplify notation, we write $c=c_n$ and $d=d_n$.

Since we know that the adjustment process terminated at row $r+n-1$, we have that $(b,a,c)$ is a counterclockwise triple in $\sigma$ if and only if $(b,a,d)$ is in $\mathcal{T}_i^{(r)}(\sigma)$. 
Without loss of generality, assume $c>d$. First suppose both triples are counterclockwise.  We represent the order on the entries $a,b,c,d$ by the circle 
\begin{tikzpicture}
\circleS cb{a}d
\end{tikzpicture}, where the entries (read in counter-clockwise order starting from the smallest entry) appear in order from smallest to largest.

The relative position of $f$ gives 4 possible cyclic orders, which we consider in the table below. In this table, we check that indeed the total number of counterclockwise triples in the two pairs is preserved. 

\begin{center}
\begin{tabular}{|c||c|c||c|c|}
\hline
 &  \multicolumn{2}{c||}{$\sigma$} & \multicolumn{2}{c|}{$\mathcal{T}_i^{(r)}(\sigma)$}\\
& \tableau{a&&f\\c}&\tableau{b&&f\\d}&\tableau{a&&f\\ d}&\tableau{b&&f\\c}\\
\hline\hline
\begin{tikzpicture}
\circleT cba{\bf{f}}d
\end{tikzpicture}
& clock&clock&clock&clock\\
\hline
\begin{tikzpicture}
\circleT cb{\bf{f}}ad
\end{tikzpicture}
& counter&clock&counter&clock\\
\hline\hline
\begin{tikzpicture}
\circleT {\bf{f}}badc
\end{tikzpicture}
&counter&counter&counter&counter\\
\hline
\begin{tikzpicture}
\circleT cbad{\bf{f}}
\end{tikzpicture}
&clock&counter&counter&clock\\
\hline
\end{tabular}
\end{center}

The case where both triples are clockwise, implying the cyclic order \begin{tikzpicture}
\circleS ca{b}d
\end{tikzpicture}, is symmetric. By construction, no other counterclockwise triples are introduced or removed, so in all the possible cases, $\inv$ changes by exactly one based on the orientation of the triple $T$, thus completing the proof of the lemma.
\end{proof}

\begin{remark}\label{rem:triples}
Let $\sigma$ be a filling, and let $T$ be a triple in $\sigma$ consisting of three cells $(j,r), (i,r), (i,r-1)$, where $i<j$ and columns $i$ and $j$ have different lengths. Then a consequence of the proof of \cref{lem:coinv} is that for any operator $\mathcal{T}_u^{(r)}$, either
\begin{itemize}
\item $T$ is a counterclockwise triple in both $\sigma$ and $\mathcal{T}_u^{(r)}(\sigma)$, or
\item $T$ is a clockwise triple in both $\sigma$ and $\mathcal{T}_u^{(r)}(\sigma)$.
\end{itemize}
\end{remark}

\subsection{Sequences of operators and positive distinguished subexpressions}

So far we have defined operators $\mathcal{T}_i^{(r)}$ that act on two adjacent columns $i$ and $i+1$ at row $r$ of a filling. We will also need to consider sequences of these operators.  

Recall that the symmetric group $S_n$ is generated by the simple transpositions $\{s_1,\dots,s_{n-1}\}$, subject to the relations $s_i^2 = e$ and $s_i s_{i+1} s_i = s_{i+1} s_i s_{i+1}$, where $s_i\in S_n$ is the simple transposition which swaps $i$ and $i+1$, and $e$ is the identity. Given any $w\in S_n$, we say that $w=s_{i_t} \dots s_{i_2} s_{i_1}$ is reduced if there are no factorizations of $w$ into fewer than $t$ simple transpositions.  We often refer to a fixed reduced expression $s_{i_t} \dots s_{i_2} s_{i_1}$ for the permutation $w$ as $\mathbf{w}$.  We call $t$ the \emph{length} of $w$ and denote it $\ell(w) = t$.

\begin{definition}
Fix $n$ and suppose that $\mathbf{w} = 
s_{i_t} \cdots s_{i_1}$ is a reduced expression for a permutation $w\in S_n$.
Then we let $\mathcal{T}^{(r)}_{\mathbf{w}}$ denote the composition of operators
\[
\mathcal{T}_{i_t}^{(r)} \dots  \mathcal{T}_{i_1}^{(r)}.
	\]
 \end{definition}

Note that while each operator $\mathcal{T}_i^{(r)}$ is an involution, the operators do not satisfy braid relations; i.e. in general,  $\mathcal{T}_i^{(r)}
 \mathcal{T}_{i+1}^{(r)}
 \mathcal{T}_i^{(r)}(\sigma)
  \neq \mathcal{T}_{i+1}^{(r)}
 \mathcal{T}_{i}^{(r)}
 \mathcal{T}_{i+1}^{(r)}(\sigma).$  
So the operator $\mathcal{T}^{(r)}_{\mathbf{w}}$ depends on the reduced expression $\mathbf{w}$, not just the permutation $w$.

Since we will want to associate one operator to each permutation, we need to choose a canonical
reduced expression for each permutation.  For this,
we use the notion of the \emph{positive distinguished subexpression} (or PDS)
of a reduced expression, as defined in \cite{MR}.

\begin{defn} \label{def:posdistsubexpression} Let $v$ and $w$ be
permutations in $S_n$ such that $v \leq w$ in the (strong) Bruhat order. Fix $\mathbf{w}=s_{i_t} \cdots s_{i_1}$ a reduced expression for $w$. The \emph{positive distinguished subexpression} (PDS) for $v$ in $\mathbf{w}$ is a reduced expression $\mathbf{v}=v_t \dots v_1$ where $v_j \in \{s_{i_j}, e\}$, which is defined as follows.  First we initialize $v_{(0)}=v$, and then we inductively set 
\[ v_{(j)}=\left\{
\begin{array}{ll}
	v_{(j-1)}s_{i_j} & \text{ if } v_{(j-1)}s_{i_j}<v_{(j-1)} \\
v_{(j-1)} & \text{ otherwise}.
\end{array}
\right.
\]
Correspondingly we have 
$v_j=s_{i_j}$ if $v_{(j-1)}s_{i_j}<v_{(j-1)}$ and $v_j=e$ otherwise.  
\end{defn}
Note that informally, the PDS for $v$ is the rightmost reduced subexpression for $v$ in $\mathbf{w}$, where we choose its simple transpositions greedily from right to left.

\begin{example}
Consider $v=(2, 5, 1, 4, 3)$ (written in one-line notation), $w=(5, 4, 3, 2, 1)$, and the following reduced expression $\mathbf{w}$ for $w$, where the PDS for $v$ is in bold and the steps used to determine this PDS are listed in Figure~\ref{table:PDS}.
\[\mathbf{w}=s_1 s_2 s_1 \mathbf{s_3} s_2 \mathbf{s_1} \mathbf{s_4} \mathbf{s_3} \mathbf{s_2} s_1.
\]

\begin{figure}
\begin{tabular}{|c|c|c|c|}
\hline
j & $v_{(j-1)} s_{i_j}$ & $v_{(j)}$ & $v_j$ \\
\hline
\hline
0 &  N/A & $(2,5,1,4,3)$ & N/A \\
\hline
1 & $(5,2,1,4,3)$ & $(2,5,1,4,3)$ & $e$ \\
\hline
2 & $(2,1,5,4,3)$ & $(2,1,5,4,3)$ & $s_2$ \\
\hline
3 & $(2,1,4,5,3)$ & $(2,1,4,5,3)$ & $s_3$ \\
\hline
4 & $(2,1,4,3,5)$ & $(2,1,4,3,5)$ & $s_4$ \\
\hline
5 & $(1,2,4,3,5)$ & $(1,2,4,3,5)$ & $s_1$ \\
\hline
6 & $(1,4,2,3,5)$ & $(1,2,4,3,5)$ & $e$ \\
\hline
7 & $(1,2,3,4,5)$ & $(1,2,3,4,5)$ & $s_3$ \\
\hline
8 & $(2,1,3,4,5)$ & $(1,2,3,4,5)$ & $e$ \\
\hline
9 & $(1,3,2,4,5)$ & $(1,2,3,4,5)$ & $e$ \\
\hline
10 & $(2,1,3,4,5)$ & $(1,2,3,4,5)$ & $e$ \\
\hline
\end{tabular}
\caption{Determining the PDS for $v=(2,5,1,4,3)$ in $\mathbf{w}=s_{i_{10}}\cdots s_{i_2}s_{i_1}=s_1 s_2 s_1 s_3 s_2 s_1 s_4 s_3 s_2 s_1$.}{\label{table:PDS}}
\end{figure}

	Therefore the PDS for $v$ is $\mathbf{v} = e e e s_3 e s_1 s_4 s_3 s_2 e = s_3 s_1 s_4 s_3 s_2$. 
\end{example}

\begin{definition}
Fix a positive integer $n$ and let $\mathbf{w_0}$ be the canonical reduced
expression
$(s_1)(s_2 s_1)\ldots(s_{n-1} \dots s_2s_1)$
For each permutation $w\in S_n$, we let 
$\mathbf{w}= s_{i_k} \dots s_{i_1}$ be the PDS for $w$ in $\mathbf{w_0}$.  
Let $\PDS(n):= \{ \mathbf{w} \ \vert \ w\in S_n\}$.
We also define for each permutation $w\in S_n$ and each choice of $r$ the operator 
	\[\mathcal{T}_w^{(r)} = \mathcal{T}_{i_k}^{(r)} \dots \mathcal{T}_{i_1}^{(r)}.\]
\end{definition}

\begin{example}
When $n=3$, we have that 
$\PDS(n) = \{e, s_1, s_2, s_1 s_2, s_2 s_1, s_1 s_2 s_1\}$.
Note however that $s_2 s_1 s_2 \notin \PDS(3)$.
\end{example}

\begin{remark}\label{rem:truncation}
Note that by construction, if $s_{j_t} \dots s_{j_2}  s_{j_1}$ lies in $\PDS(n)$, 
then so does $s_{j_h} \dots s_{j_2} s_{j_1}$ for each $ t \geq h \geq 1$.  
In other words, $\PDS(n)$ is closed under truncation: if we truncate
any PDS in $\PDS(n)$
by removing simple reflections at the left, we obtain again a PDS in $\PDS(n)$.
\end{remark}

\begin{remark}\label{rem:spanning}
Suppose  we construct a graph whose vertices are labeled by the permutations in $S_n$, where we add an edge from $u$ to $v$ if the PDS for $v$ is obtained by adding one simple reflection to the left of the PDS for $u$.  It follows from \cref{rem:truncation} that  this graph is a tree; moreover it is a spanning tree for the $1$-skeleton of the permutohedron.  
\end{remark}

We now consider how the $\inv$ statistic changes when we perform a sequence of operators on a tableau.

\begin{lemma}\label{lem:tworow}
Let $\sigma$ denote the sorted two-row tableau

\begin{center}
 \begin{tikzpicture}[scale=.5]
\cell11{$c$} \cell12{$c$} \cell13{$\dots$}\cell14{$c$}
\cell01{$b_1$} \cell02{$b_2$} \cell03{$\dots$}\cell04{$b_n$}
\end{tikzpicture}
\end{center}
and let $\sigma'$ denote a new tableau
\begin{center}
\begin{tikzpicture}[scale=.5]
\cell11{$c$} \cell12{$c$} \cell13{$\dots$}\cell14{$c$}
\cell01{$w_1$} \cell02{$w_2$} \cell03{$\dots$}\cell04{$w_n$}
\end{tikzpicture}
\end{center}
obtained from $\sigma$ by permuting the entries $b_1,\dots, b_n$. We define the \emph{length} $L(\sigma')$ to be the length $\ell(\tilde{w})$ of a shortest permutation $\tilde{w}$ such that $(b_{\tilde{w}(1)}, \dots, b_{\tilde{w}(n)}) = (w_1,\dots, w_n)$. Then $L(\sigma') = \inv(\sigma')$.
\end{lemma}
\begin{proof}
By \cref{lem:coinv}, each time we swap two consecutive entries in the bottom row
of such a two-row tableau, we increase the number of counterclockwise triples by at most $1$.
So $\inv(\sigma') \leq \ell(\tilde{w}) = L(\sigma')$.

We prove the opposite inequality by induction. By the definition of a sorted tableau, there exists some $0 \leq h \leq n$ such that 
	\begin{equation}\label{temp}
		b_{h} \leq b_{h+1} \leq \dots \leq b_n \leq c < b_1 \leq b_2 \leq \dots \leq b_{h-1}.
	\end{equation}
If $w$ is such that $\tilde{w}=\mbox{id}$, then $\inv(\sigma) = \inv(\sigma') = 0 = \ell(\tilde{w}).$ This gives the base case.  For the inductive step, suppose that $\inv(\sigma') = k+1 \geq 1$. Then $\sigma'$ is not a sorted tableau; or equivalently, \eqref{temp} is violated.  But in that case there exists a position $i$ such that $\inv(\mathcal{T}_i^{(1)}(\sigma')) = k$. By the induction hypothesis, this implies that $L(\mathcal{T}_i^{(1)}(\sigma'))\leq k$, and hence $L(\sigma') \leq k+1$.
\end{proof}

\subsection{The family $\mathcal{F}(\sigma)$ generated by a tableau}

In this section, we associate a set of tableaux $\mathcal{F}(\sigma)$ to each 
sorted tableau $\sigma$.  This set of tableaux is obtained from $\sigma$ by applying 
sequences of operators to $\sigma$.  See \cref{ExpandedExample} for a collection of sorted tableaux and their families.

We first define the \emph{block decomposition} of a sorted tableau.

\begin{definition}\label{def:blocks}
Let $\sigma$ be a sorted tableau. A \emph{block} in row $r$ 
of $\sigma$ is a maximal-by-inclusion contiguous sequence of elements 
	$(\sigma(i,r), \sigma(i+1,r), \dots, \sigma(j,r))$ (with $i<j$) such that 
$\sigma(i,s) = \dots= \sigma(j,s)$ for every row $s<r$ below $r$ in the diagram.
\end{definition}

Since every row of $\sigma$ can be partitioned into blocks, we get a \emph{block decomposition} of $\sigma$. See \cref{fig:block} for an example of a block decomposition of a sorted tableau.

\begin{remark}
If $(\sigma(i,r), \dots, \sigma(j,r))$ is a block in row $r$, then $(\sigma(i,r+1),\dots, \sigma(j,r+1))$ is a concatenation of blocks in row $r+1$.
\end{remark}

\begin{figure}

  \centering

\begin{tikzpicture}[scale=.6]

\cell013\cell021\cell033\cell041\cell051 \cell066
\cell111 \cell123 \cell133 \cell145 \cell155 \cell161
\cell213 \cell222 \cell232 \cell242\cell252\cell262
\cell314 \cell326 \cell336 \cell346 \cell355 \cell365
\cell413 \cell423 \cell433\cell443 \cell459 \cell469

\draw (0,-4)--(0,1)--(6,1);
\draw[blue] (0.2,-3.2)--(0.2,-3.8)--(5.8,-3.8)--(5.8,-3.2)--(.2,-3.2);
\draw[blue] (0.2,-2.2)--(0.2,-2.8)--(3.8,-2.8)--(3.8,-2.2)--(0.2,-2.2);
\draw[blue] (4.2,-2.2)--(4.2,-2.8)--(5.8,-2.8)--(5.8,-2.2)--(4.2,-2.2);
\draw[blue] (1.2,-1.2)--(1.2,-1.8)--(3.8,-1.8)--(3.8,-1.2)--(1.2,-1.2);
\draw[blue] (4.2,-1.2)--(4.2,-1.8)--(5.8,-1.8)--(5.8,-1.2)--(4.2,-1.2);
\draw[blue] (0.2,-1.2)--(0.2,-1.8)--(0.8,-1.8)--(0.8,-1.2)--(0.2,-1.2);
\draw[blue] (1.2,-0.2)--(1.2,-0.8)--(3.8,-0.8)--(3.8,-0.2)--(1.2,-0.2);
\draw[blue] (4.2,-0.2)--(4.2,-0.8)--(5.8,-0.8)--(5.8,-0.2)--(4.2,-0.2);
\draw[blue] (0.2,-0.2)--(0.2,-0.8)--(0.8,-0.8)--(0.8,-0.2)--(0.2,-0.2);
\draw[blue] (5.2,.2)--(5.2,.8)--(5.8,.8)--(5.8,.2)--(5.2,.2);
\draw[blue] (3.2,.2)--(3.2,.8)--(3.8,.8)--(3.8,.2)--(3.2,.2);
\draw[blue] (4.2,.2)--(4.2,.8)--(4.8,.8)--(4.8,.2)--(4.2,.2);
\draw[blue] (1.2,.2)--(1.2,.8)--(2.8,.8)--(2.8,.2)--(1.2,.2);
\draw[blue] (0.2,.2)--(0.2,.8)--(0.8,.8)--(0.8,.2)--(0.2,.2);

\node at (-1,.5) {\rm row 5};
 \node at (-1,-.5) {\rm row 4};
\node at (-1,-1.5) {\rm row 3};
\node at (-1,-2.5) {\rm row 2};
\node at (-1,-3.5) {\rm row 1};
\end{tikzpicture}

\caption{Block decomposition of a filling} 
	\label{fig:block}
\end{figure}

\begin{definition}\label{def:Sym}
Suppose $B = (b_1,\dots,b_n)$ is a word, possibly with 
repeated entries.  
	Let \begin{equation*}
		\Sym(B) = \{w = (w_1,\dots,w_n) \ \vert \ (w_1,\dots,w_n)
	\text{ is a rearrangement of the letters of }B\}.
	\end{equation*}
	For each $w\in \Sym(B)$, let $\tilde{w}$ be the 
	shortest permutation such that 
	\[(b_{\tilde{w}^{-1}(1)},\dots, b_{\tilde{w}^{-1}(n)}) = (w_1,\dots, w_n).\]  
	Let \begin{equation*}
		\widetilde{\Sym}(B) = \{\tilde{w} \ \vert \ w\in \Sym(B)\}.
	\end{equation*}
\end{definition}

Note that $\Sym(B)$ is a set (rather than a multiset), since although there might be two different permutations of the letters of $B$ resulting in the same rearrangement, $\Sym(B)$ is the set of all \emph{distinct} such rearrangements.

\begin{example} 
In the table below, we show the element of $\widetilde{\Sym}$ and the PDS corresponding to each $w\in\Sym(B)$ for $B=(1,1,2,2)$.
\begin{center}
{\renewcommand{\arraystretch}{1.1}
\begin{tabular}{c|c|c}
$\Sym(B)$&$\widetilde{\Sym}(B)$&PDS\\
\hline
$1122$&$1234$&$e$\\
$1212$&$1324$&$s_2$\\
$1221$&$1342$&$s_2s_3$\\
$2112$&$3124$&$s_2s_1$\\
$2121$&$3142$&$s_2s_3s_1$\\
$2211$&$3412$&$s_2s_1s_3s_2$
\end{tabular}
}
\end{center}
\end{example}

\begin{remark}\label{rem:multinomial}
Let $B=(b_1,\ldots ,b_n)$, $b_1 \leq \dots \leq b_n$, and $w\in \Sym(B)$.  Then $\ell(\tilde{w})$ equals the number of inversions of $w$. Moreover, if $B$ contains $r$ distinct entries (without loss of generality suppose the entries are $1$ through $r$), i.e. $B = 1^{m_1} 2^{m_2} \dots r^{m_r}$, where $m_1+m_2+\dots + m_r = n$, then by \cite[Theorem 5.1]{FoataHan} we have
\begin{equation}
		{n \choose m_1,m_2,\dots,m_r}_t = \sum_{\tilde{w}\in \widetilde{\Sym}(B)} t^{\ell(\tilde{w})}.
	\end{equation}

\end{remark}

In what follows, 
we will often single out the $r$th row of a tableau, and permute
the elements of its blocks.  
\begin{definition}
	Suppose that  $\sigma$ has blocks $B^{(1)},\dots, B^{(\ell)}$ in 
	row $r.$
Let $\Sym(B^{(1)}) \times \dots \times \Sym(B^{(\ell)})$ denote the set of all
permutations of the $r$th row of $\sigma$, which only permute entries within a 
block. We also let 
	 \begin{equation*}
		 \widetilde{\Sym}(B^{(1)}) \times \dots \widetilde{\Sym}(B^{(\ell)}) = \{\tilde{w} \ \vert \ w\in \Sym(B^{(1)}) \times \dots \times \Sym(B^{(\ell)})\}.
	\end{equation*}
\end{definition}

In our definition of the family $\mathcal{F}(\sigma)$ below, we will first discuss the case where the diagram of $\sigma$ is a rectangle.  In the general case, when a sorted tableau $\sigma$ is a concatenation from left to right of rectangular sorted tableaux $\sigma_1, \dots, \sigma_k$ (all of different heights), the elements of $\mathcal{F}(\sigma)$ are precisely the concatenations of the elements of $\mathcal{F}(\sigma_1), \dots, \mathcal{F}(\sigma_k)$.

\begin{definition}[The family of a rectangular sorted tableau]\label{def:family}
Suppose that $\sigma$ is a rectangular sorted tableau with $m$ rows and $n$ columns (where as usual, rows are labeled from $1$ to $m$ from bottom to top, and columns are labeled $1$ to $n$ from left to right). Set $\mathcal{F}^{(m)}(\sigma) = \{\sigma\}$. For $r$ from $m-1$ to $0$ we define a set $\mathcal{F}^{(r)}(\sigma)$ of tableaux. These tableaux will have the property that their entries in rows $1, 2,\dots, r-1$ agree precisely with the corresponding entries of $\sigma$.  

For each $\tau\in \mathcal{F}^{(r)}(\sigma)$, consider the block decomposition of the $r$th row of $\tau$.  Denote the entries and blocks of row $r$ by $b_1,\dots,b_n$ and $B^{(1)},\dots, B^{(\ell)}$, respectively. We set 
	\[
	\mathcal{F}^{(r-1)}(\sigma) = 
	\{\mathcal{T}_{\tilde{w}}^{(r)}(\tau) \ \vert \ 
	\tau \in \mathcal{F}^{(r)}(\sigma),\  
	w \in \widetilde{\Sym}(B^{(1)}) \times \dots \times \widetilde{\Sym}(B^{(\ell)})\}.
	\]
Finally we define \[\mathcal{F}(\sigma) = \mathcal{F}^{(0)}(\sigma).\]
\end{definition}

\begin{remark}
Note that the tableaux in $\mathcal{F}(\sigma)$ are naturally the labels of a tree rooted at $\sigma$, which we call the \emph{family tree} of $\sigma$.  Here we draw an edge from tableau $\tau$ to another tableau $\tau'$ if we have obtained $\tau'$ from $\tau$ by applying a single operator $\mathcal{T}_i^{(r)}$ for some $i$ and $r$; see \cref{ExpandedExample} for an example.  As we use the positive distinguished subexpression for every permutation (see \cref{rem:spanning}), we will get every tableau only once.  
\end{remark}

\begin{figure}
\centering
\begin{tikzpicture}[node distance=2cm]
\node[fill=gray!40] (A) at (1.5, 5) {$\tableau{1&1&1 \\ 1&2&3}$};
\node (B) at (0, 3) {$\tableau{1&1&1 \\ 2&1&3}$};
\node (C) at (3, 3) {$\tableau{1&1&1 \\ 1&3&2}$};
\node at (0, 1) {$\tableau{1&1&1 \\ 2&3&1}$};
\node at (3, 1) {$\tableau{1&1&1 \\ 3&1&2}$};
\node at (0, -1) {$\tableau{1&1&1 \\ 3&2&1}$};

\node at (2.5,5) {$t^0$};
\node at (1,3) {$t^1$};
\node at (4,3) {$t^1$};
\node at (1,1) {$t^2$};
\node at (4,1) {$t^2$};
\node at (1,-1) {$t^3$};

 \draw[->] (A)--(B)  node[midway,left] {$\mathcal{T}_1^{(1)}$};
 \draw[->] (A)--(C) node[midway,right] {$\mathcal{T}_2^{(1)}$};
 \draw[->] (0,2.5)--(0,1.5)  node[midway,left] {$\mathcal{T}_2^{(1)}$};
 \draw[->] (3,2.5)--(3,1.5) node[midway,right] {$\mathcal{T}_1^{(1)}$};
   \draw[->] (0,0.5)--(0,-0.5)  node[midway,left] {$\mathcal{T}_1^{(1)}$};

 \node[fill=gray!40] at (6, 5) {$\tableau{1&1&2 \\ 1&1&3}$};
\node at (6, 3) {$\tableau{1&2&1 \\ 1&3&1}$};
\node at (6, 1) {$\tableau{2&1&1 \\ 3&1&1}$};

\node at (7,5) {$t^2$};
\node at (7,3) {$t^3$};
\node at (7,1) {$t^4$};

 \draw[->] (6,4.5)--(6,3.5)  node[midway,left] {$\mathcal{T}_2^{(2)}$};
 \draw[->] (6,2.5)--(6,1.5) node[midway,left] {$\mathcal{T}_1^{(2)}$};

\node[fill=gray!40](G) at (10.5, 5) {$\tableau{3&1&1 \\ 1&1&2}$};
\node(H) at (9, 3) {$\tableau{1&3&1 \\ 1&1&2}$};
\node(I) at (12, 3) {$\tableau{3&1&1 \\ 1&2&1}$};

\node at (12, 1) {$\tableau{3&1&1 \\ 2&1&1}$};
\node at (9, 1) {$\tableau{1&3&1 \\ 1&2&1}$};
\node at (9, -1) {$\tableau{1&3&1 \\ 2&1&1}$};

\node at (11.6,5) {$qt^0$};
\node at (10,3) {$qt^1$};
\node at (10,1) {$qt^2$};
\node at (10,-1) {$qt^3$};
\node at (13,3) {$qt^1$};
\node at (13,1) {$qt^2$};

 \draw[->] (G)--(H) node[midway,left] {$\mathcal{T}_1^{(2)}$};
 \draw[->] (G)--(I) node[midway,right] {$\mathcal{T}_2^{(1)}$};

 \draw[->] (9,2.5)--(9,1.5) node[midway,left] {$\mathcal{T}_2^{(1)}$};
 \draw[->] (9,0.5)--(9,-0.5) node[midway,left] {$\mathcal{T}_1^{(1)}$};
  \draw[->] (12,2.5)--(12,1.5) node[midway,right] {$\mathcal{T}_1^{(1)}$};

\end{tikzpicture}
\vspace{0.2in}

\begin{tikzpicture}[node distance=2cm]
\node[fill=gray!40] (A) at (1.5, 5) {$\tableau{2&3&1 \\ 1&1&1}$};
\node (B) at (0, 3) {$\tableau{3&2&1 \\ 1&1&1}$};
\node (C) at (3, 3) {$\tableau{2&1&3 \\ 1&1&1}$};
\node at (0, 1) {$\tableau{3&1&2 \\ 1&1&1}$};
\node at (3, 1) {$\tableau{1&2&3 \\ 1&1&1}$};
\node at (0, -1) {$\tableau{1&3&2 \\ 1&1&1}$};

\node at (2.7,5) {$q^2t^0$};
\node at (1.1,3) {$q^2t^1$};
\node at (4.1,3) {$q^2t^1$};
\node at (1.1,1) {$q^2t^2$};
\node at (4.1,1) {$q^2t^2$};
\node at (1.1,-1) {$q^2t^3$};

 \draw[->] (A)--(B)  node[midway,left] {$\mathcal{T}_1^{(2)}$};
 \draw[->] (A)--(C) node[midway,right] {$\mathcal{T}_2^{(2)}$};
 \draw[->] (0,2.5)--(0,1.5)  node[midway,left] {$\mathcal{T}_2^{(2)}$};
 \draw[->] (3,2.5)--(3,1.5) node[midway,right] {$\mathcal{T}_1^{(2)}$};
  \draw[->] (0,0.5)--(0,-0.5)  node[midway,left] {$\mathcal{T}_1^{(2)}$};

  \node[fill=gray!40] (D) at (7.5, 5) {$\tableau{2&1&1 \\ 1&1&3}$};
\node (E) at (6, 3) {$\tableau{1&2&1 \\ 1&1&3}$};
\node (F) at (9, 3) {$\tableau{2&1&1 \\ 1&3&1}$};
\node at (6, 1) {$\tableau{1&1&2 \\ 1&3&1}$};
\node at (6, -1) {$\tableau{1&1&2 \\ 3&1&1}$};
\node at (9, 1) {$\tableau{1&2&1 \\ 3&1&1}$};

\node at (8.6,5) {$qt^0$};
\node at (7,3) {$qt^1$};
\node at (10,3) {$qt^1$};
\node at (7,1) {$qt^2$};
\node at (10,1) {$qt^2$};
\node at (7,-1) {$qt^3$};

 \draw[->] (D)--(E)  node[midway,left] {$\mathcal{T}_1^{(2)}$};
 \draw[->] (D)--(F) node[midway,right] {$\mathcal{T}_2^{(1)}$};
 \draw[->] (6,2.5)--(6,1.5)  node[midway,left] {$\mathcal{T}_2^{(1)}$};
 \draw[->] (6,0.5)--(6,-0.5)  node[midway,left] {$\mathcal{T}_1^{(1)}$};
 \draw[->] (9,2.5)--(9,1.5) node[midway,right] {$\mathcal{T}_1^{(1)}$};

 \node[fill=gray!40] at (12, 5) {$\tableau{1&1&3 \\ 1&1&2}$};
\node at (12, 3) {$\tableau{1&1&3 \\ 1&2&1}$};
\node at (12, 1) {$\tableau{1&1&3 \\ 2&1&1}$};

\node at (13.1,5) {$qt^2$};
\node at (13,3) {$qt^3$};
\node at (13,1) {$qt^4$};

 \draw[->] (12,4.5)--(12,3.5)  node[midway,left] {$\mathcal{T}_2^{(1)}$};
 \draw[->] (12,2.5)--(12,1.5) node[midway,left] {$\mathcal{T}_1^{(1)}$};
\end{tikzpicture}

\caption{Sorted tableaux $\sigma$ of  $\dg((2,2,2))$ such that $x^\sigma=x_1^4 x_2 x_3$ with their families and their statistics $t^{\inv} q^{\maj}$}
\label{ExpandedExample}
\end{figure}
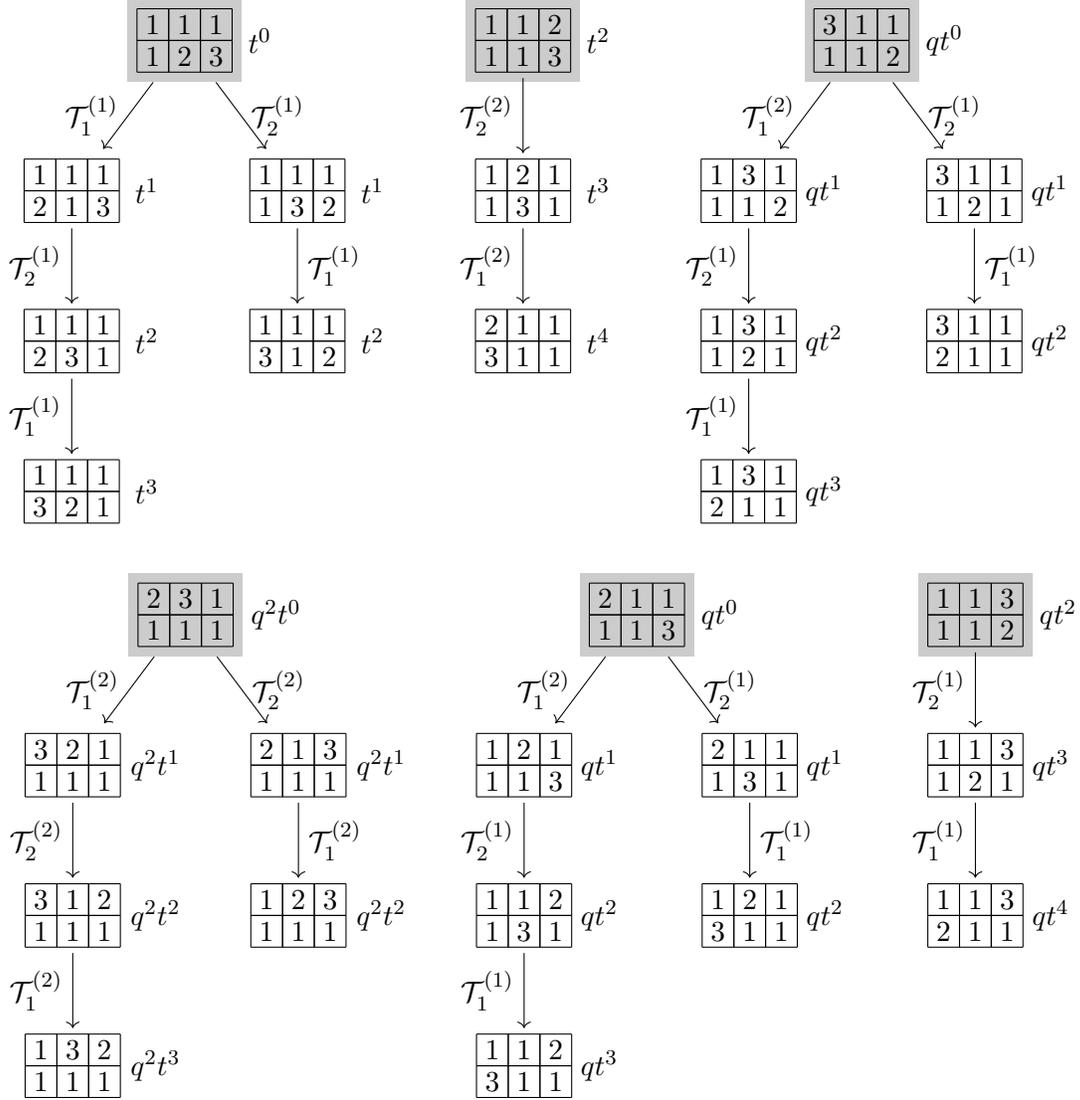

\begin{example}
Consider the sorted tableau 
	$\sigma = \tableau{2 & 1 & 1\\ 1 & 1 & 3}$
in \cref{ExpandedExample}. It has the following block decomposition.

\begin{center}
\begin{tikzpicture}[scale=.6]
\cell01{$2$} \cell02{$1$} \cell03{$1$}
\cell11{$1$} \cell12{$1$} \cell13{$3$}
\draw[blue] (0.2,-.2)--(0.2,-.8)--(2.8,-.8)--(2.8,-.2)--(.2,-.2);
\draw[blue] (2.2,.2)--(2.2,.8)--(2.8,.8)--(2.8,.2)--(2.2,.2);
\draw[blue] (0.2,.2)--(0.2,.8)--(1.8,.8)--(1.8,.2)--(0.2,.2);
\node at (-1,.5) {\rm row 2};
 \node at (-1,-.5) {\rm row 1};
\end{tikzpicture}
\end{center}

Row $2$ has two blocks, $B^{(1)} = (2,1)$ and $B^{(2)} = (1)$. $\Sym(B^{(1)}) \times \Sym(B^{(2)}) = \{(2,1,1), (1,2,1)\}$. For $w = (2,1,1)$ we have $\tilde{w} = e$;  for $w = (1,2,1)$ we have $\tilde{w} = s_1$. We set $\mathcal{F}^{(2)}(\sigma)=\{\sigma\}$ and obtain the elements of $\mathcal{F}^{(1)}(\sigma)$ by applying $\mathcal{T}_{e}^{(2)}$ and $\mathcal{T}_{s_1}^{(2)}$ to $\sigma$ to get $\mathcal{F}^{(1)}(\sigma) = \{\sigma,\ $\tableau{1 & 2 & 1\\ 1 & 1 & 3}$\ \}$.

	Row $1$ has one block, $B^{(1)} = (1,1,3)$. We have $\Sym(B^{(1)}) = \{(3,1,1), (1,3,1), (1,1,3)\}.$ For $w = (1,1,3)$ we have $\tilde{w} = e$;  for $w = (1,3,1)$ we have $\tilde{w} = s_2$; and for $w = (3,1,1)$ we have $\tilde{w} = s_1 s_2$.  So we obtain the elements of $\mathcal{F}^{(0)}(\sigma)$ by applying $\mathcal{T}_{e}^{(1)}$, $\mathcal{T}_{s_2}^{(1)}$ and $\mathcal{T}_{s_1 s_2}^{(1)}$ to each of the elements of $\mathcal{F}^{(1)}(\sigma)$. This is shown in \cref{ExpandedExample}.
\end{example}

\begin{definition}[The family of an arbitrary sorted tableau]\label{def2:family}
Suppose that $\sigma$ is a sorted tableau which is a concatenation from left to right of rectangular sorted tableaux $\sigma_1, \dots, \sigma_k$ (all of different heights).  In this case we define $\mathcal{F}(\sigma)$ to be precisely the concatenations (from left to right) of the elements of $\mathcal{F}(\sigma_1), \dots, \mathcal{F}(\sigma_k)$. 
\end{definition}

\begin{lemma}\label{lem:maj2}
	Let $\sigma$ be a sorted tableau, and choose any $\tau \in \mathcal{F}(\sigma)$.  Then $\maj(\tau) = \maj(\sigma)$.
\end{lemma}
\begin{proof}
	This follows from \cref{lem:maj} and the definition of the family $\mathcal{F}(\sigma)$.
\end{proof}

\subsection{The proof of \cref{thm:MacPerm2}}

In this section we will prove \cref{thm:MacPerm2}. Once we have verified \cref{prop:partition} and \cref{lem:genfun} below, the theorem will follow from \cref{thm:HHL}.

 \begin{proposition}\label{prop:partition}
Fix a partition $\lambda$.  Then
	 \[\bigsqcup_{\sigma\in \ST(\lambda)} 
	 \mathcal{F}(\sigma) = \bigcup_{\tau: \dg(\lambda) \to \Z^+} \tau.\]
	 \end{proposition}

\begin{proposition}\label{lem:genfun}
	Let $\sigma$ be a sorted tableau.  Then 
	\[\sum_{\tau \in \mathcal{F}(\sigma)} q^{\maj(\tau)} t^{\inv(\tau)} = 
	q^{\maj(\sigma)} t^{\inv(\sigma)} \perm_t(\sigma).\]
\end{proposition}

To prove \cref{lem:genfun}, note that by \cref{lem:maj2}, it 
suffices to prove:

\begin{proposition}\label{lem:genfun2}
Let $\sigma$ be a sorted tableau.  Then 
	\begin{equation}
	\label{eq:crucial} 
		\sum_{\tau \in \mathcal{F}(\sigma)}  t^{\inv(\tau)} = 
	t^{\inv(\sigma)} \perm_t(\sigma).
	\end{equation}
\end{proposition}

Note that \cref{lem:genfun2} implies in particular that 
		\[|\mathcal{F}(\sigma)| = \perm_{t=1}(\sigma).\]

We prove \cref{prop:partition} by explaining how to go from an arbitrary 
filling $\tau$ of $\dg(\lambda)$ to a unique sorted tableau $\sigma$ such that 
$\tau \in \mathcal{F}(\sigma)$.  We prove \cref{lem:genfun2} by using the operators $\mathcal{T}$.

\subsubsection{The proof of \cref{prop:partition}}

The following simple lemma is the key step in proving \cref{prop:partition}.
\begin{lemma}\label{def:polyqueueorder}
Let $\tau$ be a filling of an $m \times n$ rectangle, such that the restriction $\tau \vert_1^{r-1}$ of $\tau$ to rows $1, 2, \dots, r-1$ (its bottom $r-1$ rows) is a sorted tableau. Let $w_1, \dots, w_n$ be the entries in row $r$. Then there is a unique rearrangement of $w_1,\dots, w_n$, say $b = (b_1,\dots, b_n)$, such that when we add $b$ as a new row at the top of $\tau \vert_1^{r-1}$, we get a sorted tableau. We refer to the replacement of $w_1,\dots, w_n$ by $b_1, \dots, b_n$ in $\tau$ as \emph{putting row $r$ into sorted order}.
\end{lemma}
\begin{proof}
Since the bottom $r-1$ rows of $\tau$ form a sorted tableau, we have a block decomposition of row $r$, see \cref{def:blocks}. Now we just take the entries $w_1,\dots,w_n$ from row $r$, and within each block $B$, we sort the entries according to the following procedure.  

Notice that all the entries in row $r-1$ directly below a fixed block $B$ are equal to the same number, say $z$. Now we sort the entries of $B$ with respect to $z$: from left to right, start with all entries larger than $z$ in weakly increasing order, followed by all entries smaller than or equal to $z$ in weakly increasing order. After we do this for all blocks $B$, if we let $b=(b_1,\dots,b_n)$ denote the new rearrangement of the entries of row $r$, then the concatenation of $b$ at the top of $\tau \vert_1^{r-1}$ gives a sorted tableau.
\end{proof}

\begin{example}
Suppose $\tau =\tableau{7&2&4&1&1&7\\1&3&3&3&4&4}$\ , with $m=2$ and $n=6$. Then row $1$ of $\tau$ is a sorted tableau, so we take $r=2$. Then $\tau \vert_1^{r-1} = \tableau{1&3&3&3&4 & 4}$\ , the block decomposition of row $r$ induced by row $r-1$ gives $B^{(1)} = (7)$, $B^{(2)} = (2, 4, 1)$, $B^{(3)} = (1,7)$, and finally $(w_1,\dots,w_6) = (7,2,4,1,1,7)$, and $(b_1,\dots,b_6) = (7,4,1,2,7,1)$. 
\end{example}

The following algorithm starts with an arbitrary filling $\tau$ of $\dg(\lambda)$, and produces the unique sorted tableau $\sigma$ such that $\tau \in \mathcal{F}(\sigma)$. This algorithm completes the proof of \cref{prop:partition}.

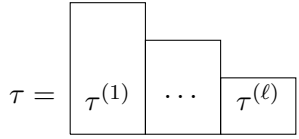
\begin{figure}
\centering
\begin{tikzpicture}[scale=0.5]
\node at (-1,.5) {$\tau =$};

\draw (0,-.5)--(0,3)--(2,3)--(2,-.5)--(0,-.5);
\node at (1,.5) {$\tau^{(1)}$};

\draw (2,-.5)--(2,2)--(4,2)--(4,-.5)--(2,-.5);
\node at (3,.5) {$\dots$};

\draw (4,-.5)--(4,1)--(6,1)--(6,-.5)--(4,-.5);
\node at (5,.5) {$\tau^{(\ell)}$};

\end{tikzpicture}
\caption{A tableau $\tau$ decomposed into maximal rectangular tableaux $\tau^{(1)},\ldots,\tau^{(\ell)}$.}\label{fig:tab_decomp}
\end{figure}

\begin{algorithm}\label{alg:reverse}
Suppose $\tau$ is an arbitrary filling of an $m\times n$ rectangle. We use the operators $\mathcal{T}_i^{(1)}$ apply \cref{def:polyqueueorder} with $r=1$ to rearrange the entries of row 1 of $\tau$ into sorted order (which in this case is just weakly increasing order). This gives us a block structure on row $2$, so we use the operators $\mathcal{T}_i^{(2)}$ to apply \cref{def:polyqueueorder} with $r=2$ to put the entries in row $2$ of the resulting tableau into sorted order.  This gives us a block structure on row $3$, and we continue, until we arrive at a unique sorted tableau which we denote by $\sigma = \sigma(\tau)$. By \cref{def:family}, it is clear that $\tau \in \mathcal{F}(\sigma)$.

	Now let $\tau:\dg(\lambda) \to \Z^+$ be an arbitrary filling of $\dg(\lambda)$. We write it as a concatenation, from left to right, of rectangular tableaux $\tau^{(1)},\dots, \tau^{(\ell)}$, each of a different height, see \cref{fig:tab_decomp}. We let $\sigma = \sigma(\tau)$ be the concatenation, from left to right, of the rectangular tableaux $\sigma(\tau^{(1)}),\dots, \sigma(\tau^{(\ell)})$. Since each of $\sigma(\tau^{(1)}),\dots, \sigma(\tau^{(\ell)})$ is a sorted tableau, so is their concatenation.  And by \cref{def2:family}, it is clear that $\tau \in \mathcal{F}(\sigma)$. 
\end{algorithm}

See \cref{fig:reverse} for an example of \cref{alg:reverse}.

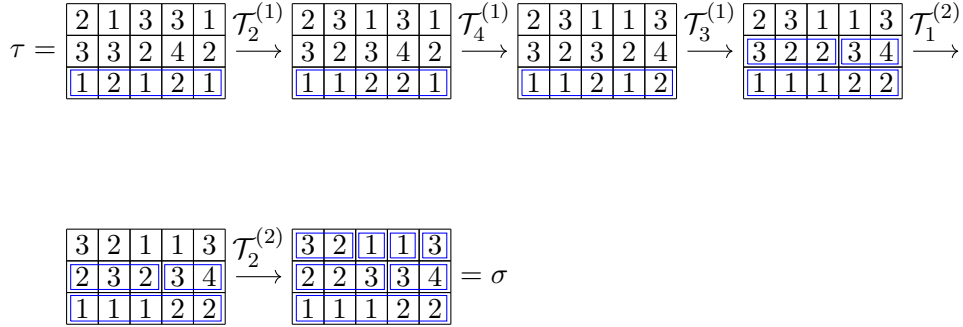
\begin{figure}
\begin{center}
\begin{tikzpicture}[node distance=2cm]
\node at (-1.5,4) {$\tau=$};
\node at (4.5,1) {$= \sigma$};

\node at (0,4) {$\tableau{ 2&1&3&3&1 \\ 3&3&2&4&2 \\  1&2&1&2&1 }$};
\node at (3,4){$\tableau{ 2&3&1&3&1 \\ 3&2&3&4&2 \\ 1&1&2&2&1 }$};
\node at (6,4){$\tableau{ 2&3&1&1&3 \\ 3&2&3&2&4 \\  1&1&2&1&2 }$};
\node at (9,4){$\tableau{ 2&3&1&1&3 \\ 3&2&2&3&4 \\ 1&1&1&2&2  }$};
\node at (0,1){$\tableau{ 3&2&1&1&3 \\ 2&3&2&3&4 \\ 1&1&1&2&2 }$};
\node at (3,1){$\tableau{ 3&2&1&1&3 \\ 2&2&3&3&4 \\ 1&1&1&2&2 }$};

\draw[->] (1.2,4)--(1.8,4)  node[midway,above] {$\mathcal{T}_2^{(1)}$};
\draw[->] (4.2,4)--(4.8,4)  node[midway,above] {$\mathcal{T}_4^{(1)}$};
\draw[->] (7.2,4)--(7.8,4)  node[midway, above] {$\mathcal{T}_3^{(1)}$};
\draw[->] (10.2,4)--(10.8,4)  node[midway,above] {$\mathcal{T}_1^{(2)}$};
\draw[->] (1.2,1)--(1.8,1)  node[midway,above] {$\mathcal{T}_2^{(2)}$};

\draw[blue] (-1,3.41)--(-1,3.75)--(1,3.75)--(1,3.41)--(-1,3.41);
\draw[blue] (2,3.41)--(2,3.75)--(4,3.75)--(4,3.41)--(2,3.41);
\draw[blue] (5,3.41)--(5,3.75)--(7,3.75)--(7,3.41)--(5,3.41);
\draw[blue] (8,3.41)--(8,3.75)--(10,3.75)--(10,3.41)--(8,3.41);
\draw[blue] (-1,.41)--(-1,.75)--(1,.75)--(1,.41)--(-1,.41);
\draw[blue] (2,.41)--(2,.75)--(4,.75)--(4,.41)--(2,.41);

\draw[blue] (8,3.83)--(8,4.16)--(9.17,4.16)--(9.17,3.83)--(8,3.83);
\draw[blue] (9.25,3.83)--(9.25,4.16)--(10,4.16)--(10,3.83)--(9.25,3.83);

\draw[blue] (-1,.83)--(-1,1.16)--(.17,1.16)--(.17,.83)--(-1,.83);
\draw[blue] (.25,.83)--(.25,1.16)--(1,1.16)--(1,.83)--(.25,.83);

\draw[blue] (2,.83)--(2,1.16)--(3.17,1.16)--(3.17,.83)--(2,.83);
\draw[blue] (3.25,.83)--(3.25,1.16)--(4,1.16)--(4,.83)--(3.25,.83);

\draw[blue] (2,1.25)--(2,1.58)--(2.75,1.58)--(2.75,1.25)--(2,1.25);
\draw[blue] (2.83,1.25)--(2.83,1.58)--(3.17,1.58)--(3.17,1.25)--(2.83,1.25);
\draw[blue] (3.25,1.25)--(3.25,1.58)--(3.58,1.58)--(3.58,1.25)--(3.25,1.25);
\draw[blue] (3.68,1.25)--(3.68,1.58)--(4,1.58)--(4,1.25)--(3.68,1.25);

\end{tikzpicture}
\end{center}
\caption{An example of the reverse algorithm is shown for the generalized tableau $\tau$ in the upper left corner to find the corresponding sorted tableau $\sigma$ in the bottom right corner. Note that we have indicated the block decomposition of a given row as soon as all rows below it are in sorted order.}\label{fig:reverse}
\end{figure}

\subsubsection{The proof of \cref{lem:genfun2}}

We start by reducing \cref{lem:genfun2} to the case where $\sigma$ is a rectangular sorted tableau.

\begin{lemma}\label{lem:shape}
If \cref{lem:genfun2} is true for sorted rectangular tableaux, then it is true for all sorted tableaux.
\end{lemma}

\begin{proof}
Let $\sigma$ be a sorted tableau that is a concatenation of two rectangular sorted tableaux $\sigma_1$ and $\sigma_2$, which have different heights. Note that $\sigma$ may contain some counterclockwise triples $(j,r)$, $(i,r)$, $(i,r-1)$ for $i<j$, where columns $i$ and $j$ have different lengths.  Let $s$ be the number of such triples.  By \cref{rem:triples}, this set of counterclockwise triples will be invariant under the operators.

We have that $\mathcal{F}(\sigma)$ is the concatenation of elements of $\mathcal{F}(\sigma_1)$ and $\mathcal{F}(\sigma_2)$, and hence if $\tau \in \mathcal{F}(\sigma)$ has the form $\tau = \tau_1 \sqcup \tau_2$ where $\tau_i \in \mathcal{F}(\sigma_i)$, then $\inv(\tau) = \inv(\tau_1)+\inv(\tau_2)+s$. Therefore the left-hand side of \eqref{eq:crucial} becomes
\begin{align*}
	\sum_{\tau \in \mathcal{F}(\sigma)}  t^{\inv(\tau)} &= 
		\sum_{\tau_1 \in \mathcal{F}(\sigma_1), 
		\tau_2\in \mathcal{F}(\sigma_2)}  t^{\inv(\tau_1)+\inv(\tau_2)+s}\\
		&=  
	t^s	
	\sum_{\tau_1 \in \mathcal{F}(\sigma_1)} t^{\inv(\tau_1)}
	\sum_{\tau_2 \in \mathcal{F}(\sigma_2)} t^{\inv(\tau_2)}.
\end{align*}
Since \eqref{eq:crucial} is true for rectangles, this quantity is equal to 
\[
t^s 
	t^{\inv(\sigma_1)} \perm_t(\sigma_1)
	t^{\inv(\sigma_2)} \perm_t(\sigma_2), 
\]
which in turn is equal to $t^{\inv(\sigma)} \perm_t(\sigma)$, as desired.
\end{proof}

We can now focus on proving \cref{lem:genfun2} for rectangular sorted tableaux.  We will prove it by induction on the number of rows.  We start with the base case.

\begin{lemma}\label{lem:base}
\cref{lem:genfun2} is true for sorted tableaux that consist of a single row.
\end{lemma}

\begin{proof}
If $\sigma$ is a sorted tableau consisting of a single row, say with entries $B=(b_1,\dots,b_n)$, then the entries are weakly increasing and $\inv(\sigma)=0$.  For any $\tau\in\mathcal{F}(\sigma)$, the entries $w$ are in $\Sym(B)$, and $\inv(\tau)=\ell(\widetilde{w})$, which puts us precisely in the situation of \cref{rem:multinomial}. The result follows from \cref{rem:multinomial}, after noting that $\perm_t(\sigma)$ is precisely the multinomial coefficient that appears in \cref{rem:multinomial}.
\end{proof}

Recall that if $\tau$ is a tableau, then we use $\tau \vert_{a}^b$ to denote the restriction of $\tau$ to rows $a$ through $b$.

\begin{lemma}\label{lem:perm}
Let $\sigma$ be a rectangular sorted tableau, whose topmost row (row $m$) has blocks $B^{(1)}, \dots, B^{(\ell)}$.  Then 
\[
\perm_t(\sigma) = \perm_t(\sigma \vert_1^{m-1}) \prod_{i=1}^{\ell} \perm_t(B^{(i)}).
\]
\end{lemma}

\begin{proof}
For $j=1,\ldots,r$, set $m_j \coloneqq |B^{(j)}|$, and let there be $n_j$ sets of identical entries whose sizes are the set $\left\{b^{(j)}_{\ell}\right\}_{1\leq \ell \leq n_j}$ with $\sum_{1\leq \ell \leq n_j} b^{(j)}_{\ell}=m_j$. By the definition of a block decomposition, this implies $\sigma$ has $\sum_{1 \leq j \leq r}m_j$ sets of identical columns whose sizes are the set $\displaystyle \uplus_{1 \leq j \leq r} \left\{b^{(j)}_{\ell}\right\}_{1\leq \ell \leq n_j}$.

Thus by the definition of $\perm_t$, the right hand side of the equation equals
\[
{n \choose m_1,\ldots,m_r}_t\  \prod_{1 \leq j \leq r} {m_j \choose b^{(j)}_1,\ldots,b^{(j)}_{n_j}}_t,
\]
which simplifies to $\perm_t(\sigma)$.
\end{proof}

\begin{proposition}\label{prop:additive}
Let $\sigma$ denote a sorted tableau with one block $B = (b_1,\dots,b_n)$ in the $r^{th}$ row.  Let $\tau$ denote a new tableau obtained from $\sigma$ by permuting the entries $b_1,\dots,b_n$	so as to get the word $w_1, w_2, \dots, w_n$. Let $\tilde{w}$ be as in \cref{def:Sym}. Then 
	\[
		\inv(\mathcal{T}^{(r)}_{\tilde{w}}(\tau)) = \inv(\tau) + \ell(\tilde{w}).
	\]
\end{proposition}
\begin{proof}
	This follows from \cref{lem:tworow} and \cref{lem:coinv}.
\end{proof}

Note that when we choose $\tau$ as above, the bottom rows $(1,\dots, r-1)$ of $\tau$ are identical to those of $\sigma$ and hence are themselves a sorted tableau.

Because the operators only affect the triples in the pair of columns where they act, \cref{cor:operator} follows from \cref{prop:additive}.

\begin{corollary}\label{cor:operator}
	We use the notation of \cref{def:family}. For each $\tau \in \mathcal{F}^{(r)}$ and $\tilde{w}\in \widetilde{\Sym}(B^{(1)}) \times \dots \times
	\widetilde{\Sym}(B^{(\ell)})$, we have that 
		\[\inv(\mathcal{T}^{(r)}_{\tilde{w}}(\tau)) = \inv(\tau) 
		+ \ell(\tilde{w}).\]
\end{corollary}

\begin{lemma}\label{lem:perm1}
Let $\sigma$ be a sorted tableau with one block $B=(b_1,\dots,b_n)$ in the $r^{th}$ row, and let $B^{(1)},B^{(2)},\dots,B^{(\ell)}$ be the blocks in row $r+1$. If we let $m_i = |B^{(i)}|$, this means that $B$ contains $\ell$ distinct entries, and $b_1 = \dots = b_{m_1}$, $b_{m_1+1} = \dots = b_{m_1+m_2}$, etc. Then we have 
	\[
		\sum_{{\tilde{w}}\in \widetilde{\Sym}(B)} t^{\ell({\tilde{w}})} = {n \choose m_1, \dots, m_{\ell}}_t = \perm_t(B).
\]
\end{lemma}
\begin{proof}
This follows from \cref{rem:multinomial}.
\end{proof}

\begin{lemma}\label{lem:rectangle}
\cref{lem:genfun2} is true for arbitrary rectangular sorted tableaux.
\end{lemma}

\begin{proof}
Let $\sigma$ be a sorted tableau which is an $m \times n$ rectangle.  We use induction on the number of rows. The base case is true by \cref{lem:base}. Suppose \cref{lem:genfun2} holds for tableaux with $(m-1)$ rows.

We need to analyze $\sum_{\tau \in \mathcal{F}(\sigma)}  t^{\inv(\tau)}$. Recall from \cref{def:family} that $\mathcal{F}(\sigma)$ consists of tableaux of the form 
\[
\T_{\tilde{w}^{(1)}}^{(1)} \circ\T_{\tilde{w}^{(2)}}^{(2)} \circ \dots \circ \T_{\tilde{w}^{(m)}}^{(m)}(\sigma).
\]
Let the top row of $\sigma$ have blocks $B^{(1)}, \dots, B^{(\ell)}$, and let $\sigma'$ denote $\sigma|_1^{m-1}$. Note that the bottom $m-1$ rows of $\T_{\tilde{w}^{(m)}}^{(m)}(\sigma)$ form the sorted tableau $\sigma'$, and row $m$ is given by $w^{(m)}\in \Sym(B^{(1)}) \times \dots \times \Sym(B^{(\ell)})$. Also, $\inv(\sigma) = \inv(\sigma')+s$, where $s$ is the number of counterclockwise triples in row $m$ of $\sigma$.

By \cref{cor:operator}, $\inv(\T_{\tilde{w}^{(m)}}^{(m)}(\sigma)) = \ell(\tilde{w}^{(m)})+ \inv(\sigma) = \ell(\tilde{w}^{(m)})+s+\inv(\sigma')$.  When we apply the operators $\T_{\tilde{w}^{(1)}}^{(1)} \circ \dots \circ \T_{\tilde{w}^{(m-1)}}^{(m-1)}$ to $\T_{\tilde{w}^{(m)}}^{(m)}(\sigma),$ they only affect $\inv$ coming from triples in rows $1$ through $m-1$. Moreover, the tableaux obtained by applying the operators $\T_{\tilde{w}^{(1)}}^{(1)} \circ \dots \circ \T_{\tilde{w}^{(m-1)}}^{(m-1)}$ to $\sigma'$ are precisely the elements of $\mathcal{F}(\sigma')$ (where we think of the rows of $\sigma'$ as being labeled from bottom to top by $1$ through $m-1$).

Therefore we have that 
\[
\sum_{\tau \in \mathcal{F}(\sigma)}  t^{\inv(\tau)} 
		= t^s \sum_{w^{(m)}} t^{\ell(\tilde{w}^{(m)})} 
		\sum_{\tau' \in \mathcal{F}(\sigma')} t^{\inv(\tau')},
\]
where the first sum on the right-hand side is over all $w^{(m)} \in \Sym(B^{(1)}) \times \dots \times \Sym(B^{(\ell)}).$
Applying \cref{lem:perm1} to the first sum on the right-hand side, and the inductive hypothesis to the second sum on the right-hand side, we find that this is equal to 
\[
t^s \prod_{i=1}^{\ell} \perm_t(B^{(i)})  \cdot t^{\inv(\sigma')} \perm_t(\sigma').
\]  
This simplifies to $\perm_t(\sigma) t^{\inv(\sigma)}$ by \cref{lem:perm} and \cref{lem:perm1}, as desired.
\end{proof}

Combining \cref{lem:shape} and \cref{lem:rectangle} yields the proof of  \cref{lem:genfun2}.

\section{A compact formula for integral  Macdonald polynomials}\label{sec:Jcompact}

In this section, we provide a compact formula for the symmetric Macdonald polynomials $P_{\mu}(X;q,t)$ and their integral form version $J_{\mu}(X;q;t)$. 
Recall that Macdonald defined the integral form $J_{\mu}(X;q,t)$ as
follows.
\begin{definition}
	Let $\mu$ be a partition.  Then 
\begin{align}
\label{JDef}
J_{\mu}(X;q,t) &= P_{\mu}(X;q,t) \text{PR}1(\mu),
\end{align}
where 
\begin{equation*}
\label{PR1}
\text{PR}1(\mu) = \prod _{s \in \dg(\mu')} (1-q^{\arm(s)}t^{\leg(s)+1}) = \prod _{s \in \dg(\mu)} (1-q^{\leg(s)}t^{\arm(s)+1}).
\end{equation*}
\end{definition}

Throughout this section we will be working with 
diagrams $\dg(\alpha)$ where $\alpha = (\alpha_1,\dots,\alpha_n)$
is a weak composition. 
For these diagrams we need to generalize the concept of triple 
defined for partition diagrams in Section \ref{sec:def}.  We use the notational conventions from \cite[Appendix A: pp.~124,~137]{Hag08}; the reader may wish to consult that source for more background and examples of these concepts.  

A {\it triple} consists of two adjacent squares $(u,r)$ and $(u,r-1)$ in the same column, together with a third square $(v,m)$ in the arm of the square $(u,r)$.  If $m=r$, so that 
$(v,r)$ is in a column to the right of $(u,r)$, we call the triple a {\it type A triple}, and it has the configuration $\tableau{\ \\ \ }\quad \tableau{\ }$\,. If $m=r-1$, so $(v,r-1)$ is in a column to the left of $(u,r)$, we call the triple a {\it type B triple}, and it has the configuration $\tableau{\\\ } \quad \tableau{\ \\  \ }$\,.  Note that all the triples from Section \ref{sec:def} were type A triples. A \emph{degenerate triple} consists of a pair of squares in the bottom row, with the rightmost of these squares in the arm of the leftmost (i.e. pairs of entries in row $1$, with the column of the rightmost square not higher than the column of the leftmost).
Two cells \emph{attack} each other if either the cells are in the same row, or they are in adjacent rows, and in different columns, with the rightmost cell in a row strictly above the other cell.  A filling is \emph{nonattacking} if it does not contain any two cells that both attack each other and contain the same number.

Given a filling $\sigma$ of   $\dg(\alpha)$,
 the entries in each triple of squares
carry with them an orientation defined as follows.  
To find the orientation of a triple, if all the entries in the three squares are distinct, start at the smallest one and go in a circular manner to
the next smallest and then the largest.  If you went in a clockwise direction, it is a clockwise triple, otherwise it is a counterclockwise triple.  For triples with repeated entries, when comparing equal entries the one which occurs first in the \emph{reading order} is viewed as being smaller. 
(Here the {\it reading word} of the filling $\sigma$ is the sequence of numbers obtained by reading the entries of the filling across rows, left-to-right, starting with the top row and working downwards.  The \emph{reading order} of the entries is the order induced by the reading word; i.e. if an entry $a$ appears before an entry $b$ in the reading word, then $a$ is said to be \emph{smaller} or \emph{earlier} in reading order than $b$.)

We call a triple a \emph{coinversion triple} if it is either a type A clockwise triple or type B counterclockwise triple. (All fillings in this section will be nonattacking, and repeated entries in any triple in a nonattacking filing must be directly on top of each other.  This forces any triple with repeated entries in a nonattacking filling to be an inversion triple.)  For degenerate triples, if the rightmost entry is smaller than the leftmost it is an inversion triple, otherwise it is a coinversion triple.  Some of our diagrams will contain a {\it basement}, i.e. a row of $n$ squares below $\text{dg}(\alpha)$  filled with a permutation in $S_n$.   We  use the notation
 $\text{dg}^{\beta}(\alpha)$ to refer to 
$\text{dg}(\alpha)$ with basement $\beta$.   We make the convention 
that squares in the basement count as being in the arm of squares in 
$\text{dg}(\alpha)$, and so can be part of triples.  But we require 
that the top square(s) of a triple or degenerate triple
must be in $\text{dg}(\alpha )$ (above the basement).

One can also define coinversion triples without the concept of an orientation using inequalities as follows.   
\begin{itemize}
\item Given a type A triple  $(v,r)$, $(u,r)$, and $(u,r-1)$ for $u<v$, set $a=\sigma((v,r))$, $b=\sigma((u,r))$, and $c=\sigma((u,r-1))$.
\item Given a type B triple $(v,r-1)$, $(u,r)$, and $(u,r-1)$ for $u>v$. Set $a=\sigma((v,r-1))$, $b=\sigma((u,r))$, and $c=\sigma((u,r-1))$.
\end{itemize}
Then a type A or B triple is a coinversion triple 
if one of the following occurs:
\begin{align*}
a \leq c < b\ \ {\rm or}\ \ 
c < b \leq a \ \ {\rm or}\ \ 
b \leq a \leq c.
\end{align*}

\begin{theorem}[{\cite[Corollary A.11.1]{Hag08}}]
	The integral form Macdonald polynomial is given by 
	\begin{align} 
\label{Jmu3}
J_{\mu}(X;q,t) =
(1-t)^{\ell (\mu)} \sum_{ \text{nonattacking fillings $\sigma$ 
of $\dg(\mu)$ }} x^{\sigma}
q^{\text{maj}(\sigma)} 
t^{\coinv(\sigma)} \\
\notag
\times \prod _{\substack{s \in \dg(\mu) \\ \text{$s$ not in row $1$}\\ \sigma (s) = \sigma (\South(s))} }
(1- q^{\leg(s)+1}t^{\arm(s)+1})
\prod _{ \substack{s \in \dg(\mu) \\ \text{$s$ not in row $1$}\\ \sigma (s) \ne \sigma ( \South(s) ) } }(1-t),
\end{align}
	where $\ell (\mu)$ is the number of (nonzero) parts of $\mu$,
	and for a cell $s$ not in row $1$,  $\South(s)$ denotes the cell directly below $s$ in the same column as $s$.   Here the sum is over all nonattacking fillings of $\dg(\mu)$ (there is no basement in these fillings).   As usual the statistic $\maj$  is the sum of $\leg(s)+1$, over all  descents (i.e. squares $s$ above row $1$ with $\sigma (s)> \sigma(\South(s))$).
\end{theorem}

In \cite{HHL04} the authors note that the right-hand-side of \eqref{Jmu3} actually yields a correct formula for $J_{\mu}$ if we replace $\mu$ everywhere by $\alpha$, where $\alpha$ is any weak composition of $n$ into $n$ parts satisfying $\dec(\alpha) = \mu$.   In fact, the most efficient formula seems to be when 
one uses $\alpha = \text{inc}(\mu)$, in which case the formula is closely related to a formula of Lenart for $P_{\mu}(X;q,t)$ \cite{Lenart} (which he proved under the additional assumption that $\mu$ has distinct parts).

One unpleasant feature of all these formulas for $J_{\mu}$ is that for the special case $\mu=1^n$, 
\begin{equation}
\label{simple}
J_{1^n}(X;q,t)=x_1x_2\cdots x_n (1-t)(1-t^2)\cdots (1-t^n),
\end{equation}
while the formula (\ref{Jmu3}) reduces to
\begin{equation*}
 x_1x_2\cdots x_n (1-t)^n \sum_{\sigma \in S_n}t^{\coinv(\sigma)},
\end{equation*}
a sum of $n!$ terms.  

In this section, we 
will obtain a compact formula for $J_{\mu}$, see
\cref{cor:J},
 which 
in the case $\mu=1^n$, has only one term---identity \eqref{simple}.  
To prove \cref{cor:J},
we will use
a recent result of the first, third, and fifth authors \cite[Theorem 1.10 and Proposition 4.1]{CMW18},
namely 
\begin{align}
\label{newPmu}
P_{\mu}(X;q,t) =
\sum_{\alpha:\ \dec(\alpha)= \mu} E_{ \inc(\alpha) }^{ \beta (\alpha) }( X;q,t),
\end{align}
where $\beta= \beta (\alpha)$ was defined in \cref{def:weak}.
In \eqref{newPmu}, the quantity
$E_{ \inc(\alpha) }^{ \beta} ( X;q,t)$ is a \emph{permuted-basement nonsymmetric Macdonald polynomial} \cite{thesis, Al16},  described combinatorially in \eqref{formula} below.  We mention that 
 in the case 
 $\gamma = (n,n-1,\cdots ,1)$, the permuted-basement nonsymmetric Macdonald polynomial $E^{\gamma} _{(\alpha_n,\ldots ,\alpha _1)}(X;q,t)$ equals the original nonsymmetric Macdonald polynomial
$E_{\alpha}(X;q,t)$ \cite{Mac88}, while the special case $\gamma = (1, 2, \cdots ,n)$ yields 
 a family of polynomials $E^{\gamma}_{\alpha}(X;q,t)$ studied by Marshall \cite{Mar99}.

\begin{definition}   
Let $\alpha = (\alpha _1, \ldots ,\alpha _n)$ be a weak composition.  We say a nonattacking filling $\sigma$ of $\dg(\inc(\alpha))$ (with or without a basement) is \emph{ordered} if in the bottom row of the diagram of $\inc(\alpha)$, entries of $\sigma$ below columns of the same height are strictly decreasing when read left to right.   
\end{definition}
Figure \ref{nonatt3b} shows an ordered, nonattacking filling of the diagram of $\alpha=(0,0,1,2,2,2,3)$.  The $7$ coinversion triples for the filling are 
$(1,6,7)$, $(3,6,7)$, $(5,6,7)$, $(4,6,7)$, $(3,1,2)$, $(4,1,2)$, and $(5,7,3)$. 
\begin{figure}
\begin{center}
\begin{tikzpicture}[scale=.5]
\cell043
\cell147 \cell132 \cell121 \cell 115
\cell246 \cell231 \cell223 \cell215 \cell204
\end{tikzpicture}

\caption{An ordered nonattacking filling 
	of $\dg((0,0,1,2,2,2,3))$ with $\maj=3$ and $\coinv= 7$.}
\label{nonatt3b}
\end{center}
\end{figure}

\begin{proposition} 
\label{13}
For $\alpha$ a weak composition,  define
\begin{equation*}
\label{EQ}
\text{PR}2(\alpha) =  \prod_{i\ge 1} (t;t)_{m_i}
\prod_{\substack{s \in \dg(\inc(\alpha)) \\ \text{$s$ not in the bottom row} }  } (1-q^{\leg(s)+1} t^{\arm(s)+1}),
\end{equation*}
where for $i\ge 1$, $m_i$ is the number of times $i$ occurs in $\alpha$, and $(t;t)_k = (1-t)(1-t^2)\cdots (1-t^k)$. 
Then if $\mu$ is any partition, $\text{PR}1(\mu) = \text{PR}2(\inc(\mu))$.
\end{proposition}

\begin{proof}   This an easy exercise, which we illustrate for the partition $\mu = (6,6,6,6,3,3)$. Consider Figure \ref{armsAB}, which lists the arm lengths for the two diagrams $\dg(\inc(\mu))$ and $\dg(\mu)$.  If we ignore squares in the bottom row of the diagram $\dg(\inc(\mu))$ and at the top of their columns in the diagram $\dg(\mu)$, there is an obvious bijection between the factors $(1-q^{ \leg } t^{ \text{ arm}+1})$ and $(1-q^{\leg+1}t^{\text{ arm}+1})$ occurring in the definition of $\text{PR}1(\mu)$ and $\text{PR}2(\text{inc}(\mu))$, respectively.  For the factors in $\text{PR}1(\mu)$ corresponding to squares at the top of their columns in $\dg(\mu)$, all the leg lengths are $0$, and we get the remaining factor $ \prod_{i\ge 1} (t;t)_{m_i}$ occurring in $\text{PR}2(\text{inc}(\mu))$.
\end{proof}

\begin{figure}
\begin{center}
\includegraphics[scale=.4]{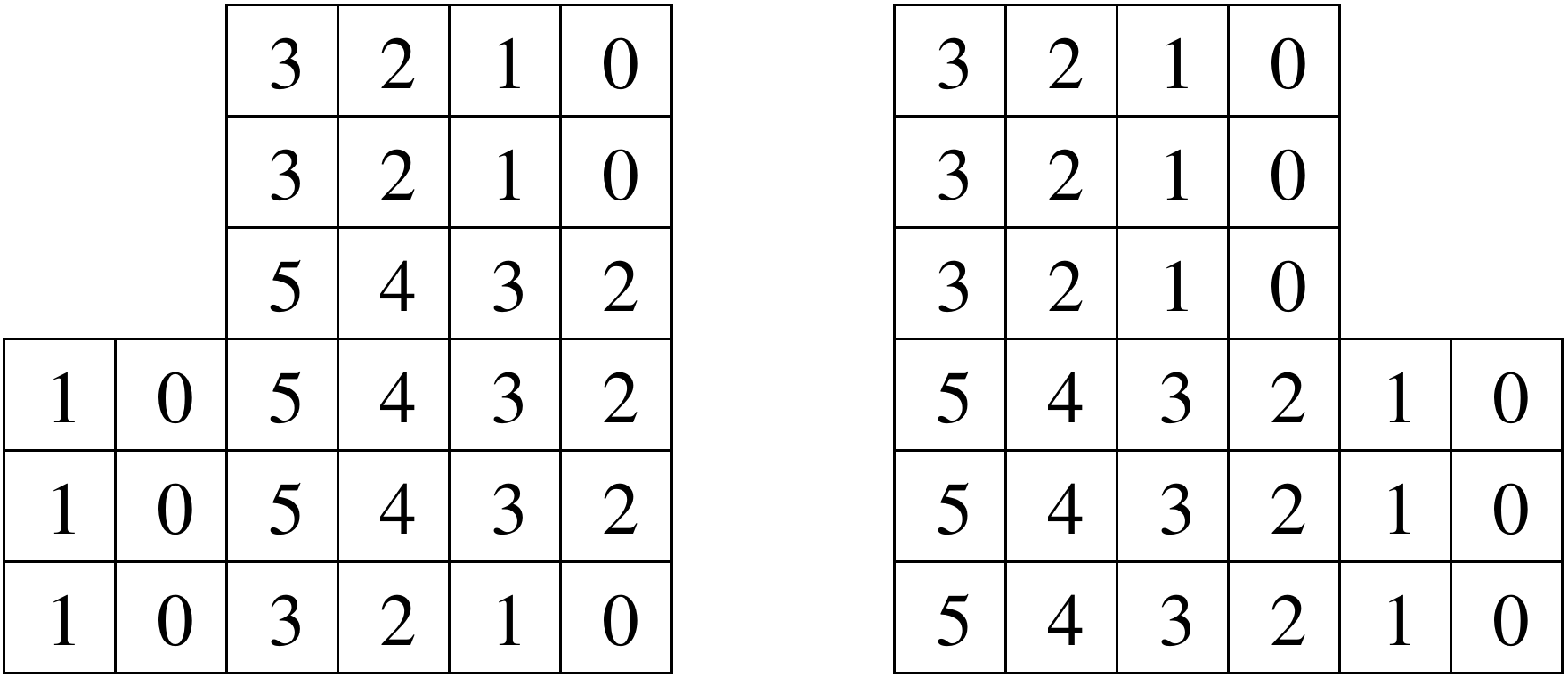}
\caption{The arm lengths for the diagrams $\dg(\inc((6,6,6,6,3,3)))$ and
$\dg((6,6,6,6,3,3))$.}
\label{armsAB}
\end{center}
\end{figure}

Recall the definition of $\beta(\alpha)$ from \cref{def:weak}.
\begin{definition}
Given a weak composition $\alpha = (\alpha_1,\dots,\alpha_n)$,
we define the \emph{integral form} version 
of $E_{\inc(\alpha)} ^{\beta (\alpha)}(X;q,t)$ as
\begin{align}
\label{integralE}
\mathcal {E}_{\inc(\alpha)} ^{\beta (\alpha)}(X;q,t) = PR2(\inc(\alpha))E_{\inc(\alpha)} ^{\beta (\alpha)}(X;q,t).
\end{align}
\end{definition}

Using the definition of 
 permuted basement nonsymmetric Macdonald polynomials 
\cite[Definition 5]{Al16},
we have:
\begin{align}
\label{formula}
E_{\inc(\alpha)}^{\beta (\alpha)}(X;q,t) = \sum_{\sigma} x^{\sigma} \wt(\sigma),
\end{align}
where the sum is over all nonattacking fillings $\sigma$ of $\dg ^{\beta (\alpha)} (\inc(\alpha))$ which use the letters $\{1,2,\dots,n\}$,  
and
\begin{align}
\label{weight}
\wt(\sigma) = q^{ \text{maj} (\sigma)} t^{ \coinv(\sigma) }\prod_
	{s:\ \sigma(s) \ne \sigma(\South(s))} 
	\frac {1-t}{1-q^{\text{leg(s)}+1}t^{\text{arm(s)}+1}}.
\end{align}

\begin{remark}
\label{non}
	For general diagrams, the product in \eqref{weight} normally includes cells $s$ in row $1$, but since the parts of  $\inc(\alpha)$ are by definition in nondecreasing order, the nonattacking condition forces the entries in row $1$ to be the same as those directly below them in the basement.  (Using 
	the definition of $\beta(\alpha)$, this will mean moreover
	that a nonattacking filling must be ordered.)  Hence for diagrams corresponding to partitions, %
we can restrict the product in \eqref{weight} to cells $s$ above row $1$. Note also that if $\sigma (s) = \sigma (\South(s)))$, then $s$ and $\South(s)$ are never part of a coinversion triple with any third cell in the arm of $s$, so such a pair will not contribute to $\coinv$ either. 
\end{remark}

It follows from (\ref{formula}), (\ref{weight}) and Remark \ref{non} that $\mathcal {E}_{\inc(\alpha)} ^{\beta (\alpha)}(X;q,t)$ 
is $\prod_{i} (t;t)_{m_i}$ times an element of $\mathbb Z [x_1,\ldots ,x_n,q,t]$.   This follows since squares $s$ in row $1$ do not contribute to the product in $\text{wt}(\sigma)$, while if any square $s$ above row $1$ satisfies $\sigma (s) \not= \sigma (\South(s))$, then the associated  factor  
$(1-q^{\leg(s)+1}t^{\arm(s)+1})$ in the denominator of 
\cref{weight}
will get cancelled by a factor in 
$\text{PR}2(\alpha)$.
In fact this argument shows that
\begin{align}
\label{IFE}
\mathcal {E}_{\inc(\alpha)} ^{\beta (\alpha)}(X;q,t) = 
\prod_{i}(t;t)_{m_i}  
\sum_{\text{ordered, nonattacking fillings $\sigma$ 
of $\dg ^{\beta (\alpha)} (\inc(\alpha))$ } } 
x^{\sigma}q^{\maj(\sigma) } t^{\coinv(\sigma )} \\
\notag
\times \prod _{\substack{s \in \dg(\inc(\alpha)),\ \text{$s$ not in row $1$} \\ 
\sigma (s) = \sigma (\South(s))} }
(1- q^{\leg(s)+1}t^{\arm(s)+1})
\prod _{ \substack{s \in \dg(\inc(\alpha)),\ \text{$s$ not in row $1$}\\ \sigma (s) \ne \sigma ( \South(s) ) } }(1-t),
\end{align}

\begin{corollary} \label{cor:J}
	The formula for $J_{\mu}$ has the following more compact version:
\begin{align} 
\label{Jmu4}
J_{\mu}(X;q,t) =
\prod_{i}(t;t)_{m_i}  
\sum_{ \text{ ordered, nonattacking fillings $\sigma$ of $\dg(\inc(\mu))$ } } 
x^{\sigma}q^{\maj(\sigma)} t^{\coinv(\sigma)} \\
\notag
\times \prod _{\substack{s \in \dg(\inc(\mu)),\ \text{$s$ not in row $1$}\\ \sigma (s) = \sigma (\South(s))} }
(1- q^{\leg(s)+1}t^{\arm(s)+1})
\prod _{ \substack{s \in \dg(\inc(\mu)),\ \text{$s$ not in row $1$}\\ \sigma (s) \ne \sigma ( \South(s) ) } }(1-t).
\end{align}
\end{corollary}
Note that \eqref{Jmu4} implies the (as far as we know) new fact that $J_{\mu}(X;q,t)$ is $\prod_{i}(t;t)_{m_i}$ times an element of
$\mathbb Z [x_1,\ldots, x_n,q,t]$.
\begin{proof}
Start by multiplying both sides of \eqref{newPmu} by $\text{PR}2(\inc(\mu))$.  The left-hand-side then becomes $J_{\mu}(X;q,t)$ by \eqref{JDef} and \cref{13}.   
The summand on the right hand side becomes $\mathcal {E}_{\inc(\alpha)} ^{\beta (\alpha)}(X;q,t)$, which by \eqref{IFE} equals the portion of \eqref{Jmu4} which has bottom row determined by $\beta(\alpha)$.
\end{proof}
By \eqref{JDef} and Proposition \ref{13}, an equivalent form of (\ref{Jmu4}) is the identity
\begin{align} 
\label{Pmu4}
P_{\mu}(X;q,t) =
\sum_{ \text{ ordered, nonattacking fillings $\sigma$ 
of $\dg(\inc(\mu))$ } } 
x^{\sigma}q^{\maj(\sigma)} t^{\coinv(\sigma)} \\
\notag
\prod _{ \substack{s \in \dg(\inc(\mu)),\ \text{$s$ not in row $1$}\\ \sigma (s) \ne \sigma ( \South(s) ) } } \frac {(1-t)}{(1- q^{\leg(s)+1}t^{\arm(s)+1})}.
\end{align}
\begin{remark}
\label{Pmuformula}
The HHL formulas for $E_{\alpha}$, $\mathcal{E}_{\alpha}$ and $J_{\mu}$ never implied a particularly nice corresponding formula for $P_{\mu}$.  For on the one hand, if you divide formula 
(\ref{Jmu3}) by $\text{PR}1(\mu)$, then by (\ref{JDef}) you get an expression for $P_{\mu}$, but the terms in $\text{PR}1(\mu)$ don't cancel the 
$(1- q^{\leg(s)+1} t^{\arm(s)+1})$ factors in (\ref{JDef}) nicely.  On the other hand, if you start by writing $P_{\mu}$ as a linear combination of
Macdonald's original $E_{\alpha}$, the coefficients are ratios of products, each with $n$ factors (see \cite[Proposition 5.3.1]{HHL04}), and the resulting formula is again not particularly nice.
By starting with (\ref{newPmu}) and using the combinatorial formula for the $E_{\text{inc}(\alpha)}^{\beta (\alpha)}$ from \cite{Al16}, we avoid these complications, and arrive at the nice formula in \eqref{Pmu4} for $P_{\mu}$.
\end{remark}

\section{A quasisymmetric Macdonald polynomial}\label{sec:qsym}

Recall that the ring of \emph{quasisymmetric functions} is a graded ring 
which  contains within it 
the ring of symmetric functions.  
The ring of quasisymmetric functions
 has multiple
distinguished bases, indexed by (strong) compositions, 
one of which is the \emph{quasisymmetric Schur
functions} 
$\text{QS}_{\gamma}(X)$ 
introduced by the second and fourth authors, together with Luoto and van Willigenburg 
\cite{HLMV09}.  (See~\cite{Mas19} for an exposition on quasisymmetric functions and some recent developments in that area of research.)  The authors showed that $\text{QS}_{\gamma}(X)$ is quasisymmetric,
and that each (symmetric) Schur function $s_{\lambda}(X)$ 
is a positive sum of quasisymmetric 
Schur functions.  In light of this, and the fact that Macdonald 
polynomials expand positively in terms of Schur polynomials, it is 
natural to ask if there is a notion of a 
\emph{Macdonald quasisymmetric polynomial} $G_{\gamma}(X; q, t)$ such that:
\begin{enumerate}
	\item[(A)]\label{a} The symmetric Macdonald polynomial $P_{\lambda}(X; q, t)$ is a positive
	     sum of Macdonald quasisymmetric polynomials; 
     \item[(B)]\label{b} $G_{\gamma}(X; q, t)$ is quasisymmetric; 
     \item[(C)]\label{c} $G_{\gamma}(X; 0, 0)$ is the quasisymmetric Schur function 
		$\text{QS}_{\gamma}(X)$;
 \item[(D)] \label{d} $G_{\gamma}(X; q, t)$ has a combinatorial formula
along the lines of the Haglund-Haiman-Loehr (HHL) formula for the $E_{\alpha}$ \cite{HHL08}, or its compact version from \cite{CMW18}.
\end{enumerate}
We show in this section that the answer to this question is yes.

Given a permutation $\sigma \in S_n$,  let $E_{\alpha}^{\sigma}(X;q,t)$ be the
 permuted-basement nonsymmetric Macdonald polynomial defined in \cite{thesis} and studied in  \cite{Al16,CMW18}, and let
 $F_{\alpha}(X;q,t) = E_{\inc(\alpha)}^{\beta (\alpha)}(X;q,t)$.  
For any partition $\lambda$ of $n$, from \cite{CMW18} we have that 
 \begin{align}
 \label{basic}
 P_{\lambda}(X;q,t)  = \sum_{\alpha :\ \dec(\alpha)= \lambda} F_{\alpha}(X;q,t),
 \end{align}
where 
the sum is over all weak compositions $\alpha$ whose positive parts are a 
rearrangement of the parts of $\lambda$. 

Note that if $\text{id}=(1,2,\ldots ,n)$ and  $w_0=(n,n-1,\ldots,1)$ are the identity permutation and permutation of maximal length in $S_n$,
respectively, then
$E_{\alpha}^{ \text{id} }(X;0,0)$ is the Demazure atom (introduced in~\cite{LasSch90} and developed in~\cite{Mas09}) and $E_{\alpha}^{ w_0 }(X;0,0)$ the Demazure character (see~\cite{ReiShi95}).  In the common notation for Demazure characters, 
i.e., key polynomials, one reverses the vector $\alpha$, i.e.,  the key polynomial corresponding to $\alpha$ would be $E_{(\alpha _n, \ldots ,\alpha _1)}^{w_0}(X;0,0)$.

Motivated by \eqref{basic}, we have the following definition and theorem.

\begin{theorem}
	We define the \emph{quasisymmetric Macdonald polynomial}
	$G_{\gamma}(X;q,t)$  to be
\begin{align}
\label{Gdef}
G_{\gamma}(X;q,t) &= \sum_{\alpha:\ \alpha ^{+}= \gamma}  F_{\alpha}(X;q,t) \\
&=  \sum_{\alpha:\ \alpha ^{+}= \gamma}  E_{\inc(\alpha)}^{\beta (\alpha )}(X;q,t),\nonumber
\end{align}
where the sum is over all weak compositions $\alpha$ for which $\alpha^{+}=\gamma$.
	Then $G_{\gamma}(X; q, t)$ satisfies properties 
	\rm{(A), (B), (C), and (D)}.
Moreover, we have that 
	\begin{equation} \label{tatom}
	F_{\alpha}(X; 0, t) = E_{\alpha}^{\text{id}}(X; 0, t),
\end{equation}
	where $E_{\alpha}^{\text{id}}(X; 0, t)$ is the Demazure $t$-atom~\cite{Al16}.
\end{theorem}

\begin{proof}
The fact that $G_{\gamma}(X; q, t)$ satisfies \rm{(A)} follows from \eqref{basic}. There are several combinatorial proofs that $G_{\gamma}(X; q,t)$ is quasisymmetric and hence satisfies \rm{(B)}.  One proof uses the multiline queues from \cite{CMW18}.  Here we provide a proof using \eqref{formula} and a notion of \emph{packed} nonattacking fillings.

\begin{figure}
\begin{center}
\begin{tikzpicture}[scale=.5]
\cell041 
\cell141 \cell132 \cell123
\cell242 \cell234 \cell225
\cell344 \cell335 \cell323
 \cell444 \cell435 \cell426 \cell412
\cell5{-5}{10}\cell5{-4}9\cell5{-3}8\cell5{-2}7\cell5{-1}3 \cell501 \cell512 \cell526 \cell535 \cell544
\end{tikzpicture}\hspace{.5cm}
\begin{tikzpicture}[scale=.5]
\cell042 
\cell142 \cell133 \cell124
\cell243 \cell237 \cell228
\cell347 \cell338 \cell324
 \cell447 \cell438 \cell42{10} \cell413
\cell5{-5}9\cell5{-4}6\cell5{-3}5\cell5{-2}4\cell5{-1}2 \cell501 \cell513 \cell52{10} \cell538 \cell547
\end{tikzpicture}
\caption{On the left, a packed, nonattacking filling (with a basement).  
On the right, a nonattacking filling with the same $qt$-weight but a shifted monomial weight.}\label{nonatt1b}
\end{center}
\end{figure}

Call a nonattacking filling $\sigma$ {\it packed} if the set of non-basement entries in $\sigma$ equal $\{1,2,\ldots ,k\}$ for some $k$, as in
\cref{nonatt1b} (where $k=6$).
By \cref{non} any entry in row 1 will not generate any contribution to $\wt(\sigma)$.   The contributions to 
$\wt(\sigma)$ by the numbers in the rows above row 1 depend only on the
relative order of these numbers; hence  if we replace the filling $\sigma$ by another filling $R$ using a replacement of the form $i \to w_i$ for $1\le i \le k$, where
$k$ is the maximal value occurring above the basement in $\sigma$  and $1\le w_1 < w_2 < \cdots < w_k\le n$, then the vector of exponents $V(x^{R})$ satisfies
$V(x^{R})^{+} = V( x^{\sigma})^{+}$, and $\wt(\sigma) = \wt(R)$. Note that changing the entries in the basement in columns of height zero doesn't affect the weight of $\sigma$. 
For example, the monomial weight of the packed filling (including the basement) on the left of \cref{nonatt1b} is $x_1^2x_2^3x_3^2x_4^3x_5^3x_6$, 
with vector $V = (2,3,2,3,3,1)$, while the monomial weight of the filling on the right of \cref{nonatt1b}
is $x_2^2x_3^3x_4^2x_7^3x_8^3x_{10}$, with vector $V=(0,2,3,2,0,0,3,3,0,1)$. 
Hence each packed nonattacking filling can be associated to a family of nonattacking fillings with the 
same weight.  The sum of the monomial weights of all these fillings is a monomial quasisymmetric function by definition, and it follows that $G_{\gamma}(X;q,t)$ is quasisymmetric.

To see that  $G_{\gamma}(X; q, t)$ satisfies \rm{(C)}, recall that for $\gamma$ a strong composition of $n$, $\text{QS}_{\gamma}$ is the quasisymmetric function defined by the equation
\begin{align}
\label{QS}
\text{QS}_{\gamma}(X) = \sum_{\alpha:\ \alpha ^{+}= \gamma} E_{\alpha}^{\text{id}}(X;0,0).
\end{align}
To verify \rm{(C)}, it suffices to show that $F_{\alpha}(X; 0, 0) = E_{\alpha}^{\text{id}}(X; 0, 0)$.  
	We actually prove \eqref{tatom},  that 
\begin{equation*} 
	F_{\alpha}(X; 0, t) = E_{\alpha}^{\text{id}}(X; 0, t).
\end{equation*}

To prove this, we use induction on the length of $\beta(\alpha)$ together with the action of the Hecke operators $T_i:= t - \frac{tx_i - x_{i+1}}{x_i-x_{i+1}} ({\bf 1} - s_i)$.  For the base case, if $\alpha_1 \le \alpha_2 \le \hdots \l \alpha_n$, then $\alpha$ is already in increasing order and therefore $\beta(\alpha)=\text{id}$.  This implies that $F_{\alpha}(X;0,t) = E_{\inc(\alpha)}^{\beta(\alpha)}(X;0,t) = E_{\alpha}^{\text{id}}(X;0,t),$ as desired.

For the inductive step, assume $F_{\alpha}(X;0,t) = E_{\alpha}^{\text{id}}(X; 0,t)$ for all compositions $\alpha$ such that $\ell(\beta(\alpha)) \le k$.  Corteel, Mandelshtam, and Williams~\cite[Theorem 1.22]{CMW18} prove that $T_i F_{\alpha} = F_{s_i \alpha}$ for $\alpha_i > \alpha_{i+1}$, and Alexandersson~\cite[Corollary 26]{Al16} asserts that $\tilde{\Theta}_i E_{\alpha}^{\text{id}}(X; 0, t) = E_{s_i \alpha}^{\text{id}}(X; 0, t)$ for $\alpha_i > \alpha_{i+1}$. 
As noted in \cite[page 8]{Al16}, a straightforward calculation shows that 
$T_i = \tilde{\Theta}_i$.
  Therefore we apply $T_i$ to both sides of $F_{\alpha}(X;0,t) = E_{\alpha}^{\text{id}}(X;0,t)$ to get
\begin{align}
F_{s_i \alpha}(X;0,t) &= T_i E_{\alpha}^{\text{id}} (X;0,t) \\
 &=  \tilde{\Theta}_i E_{\alpha}^{\text{id}} (X;0,t) \\
 &= E_{s_i \alpha}^{\text{id}} (X;0,t).
\end{align}
Therefore $F_{\alpha}(X;0,t) = E_{\alpha}^{\text{id}}(X; 0,t)$, as desired.  Specializing to $t=0$ finishes the proof that $G_{\gamma}(X;q,t)$ satisfies (C). 
By \eqref{formula},  $G_{\gamma}(X;q,t)$ satisfies (D).
\end{proof}

\begin{remark}
From \cite{HHL08} we have
\begin{align}\label{PtoE}
P_{\lambda}(X;q,t)  = \sum_{\alpha: \alpha ^{+} =\lambda} c_{\lambda,\alpha} (q,t)E_{\alpha}^{id}(X;q,t),
\end{align}
where $c_{\lambda,\alpha}$ is a certain explicit rational function in $q, t$.  Letting $q=t=0$ in this identity, since $c_{\lambda}(0,0)=1$, produces the formula of Lascoux-Sch\"utzenberger expressing the Schur function $s_{\lambda}$ as a sum of Demazure atoms. In view of~\eqref{QS}, this also shows that the Schur function is a sum of the $\text{QS}_{\gamma}$.  Trying the same idea in \eqref{PtoE} though, one finds the sum of the $c_{\lambda,\alpha}E_{\alpha}^{id}(X;q,t)$ over $\alpha^{+}=\gamma$ are not generally quasisymmetric.  
\end{remark}

It would be interesting to find a connection between the quasisymmetric Macdonald polynomials $G_{\gamma}(X; q, t)$ that we introduce in this paper, and other objects in the literature.  We note that our $G_{\gamma}$ are different from the duals of the noncommutative symmetric function analogues of Macdonald polynomials introduced in~\cite{BerZab05}; we also do not see a connection to the noncommutative Hall-Littlewood polynomials studied in~\cite{Hiv00}.

\printbibliography

\end{document}